\documentclass{article}

\usepackage{geometry}
\usepackage{cancel}
\usepackage{graphicx}
\usepackage{subcaption}
\usepackage{adjustbox}

\usepackage{algorithmicx}
\usepackage[ruled,vlined]{algorithm2e}
\usepackage{algpseudocode,float}
\algblock{ParFor}{EndParFor}

\usepackage{mathrsfs}  
\usepackage{sidecap}
\usepackage{amsfonts}
\usepackage{amssymb}
\usepackage{amsmath}

\usepackage{bm}

\usepackage{amsthm}
\usepackage{comment}
\usepackage{xcolor}
\usepackage{relsize}
\usepackage{multirow}
\usepackage{multicol}
\usepackage{diagbox}
\usepackage{overpic}
\usepackage{authblk}

\newcommand{\R}{{\mathbb R}}

\usepackage{amssymb}

\usepackage{cleveref}
\numberwithin{equation}{section}

\newtheorem{lemma}{Lemma}

\newtheorem{proposition}{Proposition}
\newtheorem{definition}{Definition}

\date{}

\title{Automatic parameter selection for the TGV regularizer in image restoration under Poisson noise}

\newcommand{\D}{\mathrm{D}}
\newcommand{\A}{\mathrm{A}}

\newcommand{\I}{\mathrm{I}}

\newcommand{\G}{\mathrm{G}}
\newcommand{\argmax}{\arg\max}
\usepackage{subcaption}

  \definecolor{cosmiclatte}{rgb}{1.0, 0.97, 0.91}
\definecolor{azure}{rgb}{0.94, 1.0, 1.0}
\definecolor{beaublue}{rgb}{0.74, 0.83, 0.9}
\usepackage{color, colortbl}
\definecolor{Gray}{rgb}{0.75,0.75,0.75}
\definecolor{apricot}{rgb}{0.98, 0.81, 0.69}
\definecolor{almond}{rgb}{0.94, 0.87, 0.8}
\definecolor{LightCyan}{rgb}{0.88,1,1}

\author[1]{Daniela di Serafino\thanks{daniela.diserafino@unina.it}}
\author[1]{Monica Pragliola\thanks{monica.pragliola@unina.it}}
\affil[1]{Department of Mathematics and Applications, University of Naples Federico II, Naples, Italy}

\begin{document}
	
\maketitle

\begin{abstract}
	We address the image restoration problem under Poisson noise corruption. The Kullback-Leibler divergence, which is typically adopted in the variational framework as data fidelity term in this case, is coupled with the second-order Total Generalized Variation (TGV$^2$). The TGV$^2$ regularizer is known to be capable of preserving both smooth and piece-wise constant features in the image, however its behavior is subject to a suitable setting of the parameters arising in its expression. We propose a hierarchical Bayesian formulation of the original problem coupled with a Maximum A Posteriori estimation approach, according to which the unknown image and parameters can be jointly and automatically estimated by minimizing a given cost functional. The minimization problem is tackled via a scheme based on the Alternating Direction Method of Multipliers, which also incorporates a procedure for the automatic selection of the regularization parameter by means of a popular discrepancy principle. Computational results show the effectiveness of our proposal.
\end{abstract}

\section{Introduction\label{sec:intro}}

The image restoration task under Poisson noise corruption has been extensively tackled in literature as it arises in many real-world applications involving photon counting processes, ranging from microscopy to astronomical imaging problems \cite{pois1}.

In these settings, the general image formation model takes the form
\begin{equation}\label{eq:lin_mod}
\bm{b} \;{=}\; \mathrm{Poiss}(\bm{y})\,,\quad \bm{y}\;{=}\;\bm{\mathrm{A}u}\,+\,\bm{\gamma}\,,
\end{equation}
where $\bm{\mathrm{A}}\in\R^{n\times n}$ is the blur operator, which is assumed to be known, $\bm{b}\in\mathbb{N}^n$ and $\bm{u}\in\R^n$ are the vectorized observed and uncorrupted image, respectively, $\bm{\gamma}\in\R^n$ is a non-negative background emission, and
$\mathrm{Poiss}(\bm{y})\;{:=}\;(\mathrm{Poiss}(y_1),\ldots,\mathrm{Poiss}(y_n))$, with $\mathrm{Poiss}(y_i)$ representing the realization of a Poisson random variable with mean $y_i$.

Typically, problem \eqref{eq:lin_mod} cannot be addressed directly by solving a linear system because of the ill conditioning of the blur operator. As a result, one seeks for an estimate $\bm{u}^*$ of $\bm{u}$ that is as close as possible to the original uncorrupted image. In the variational framework, the original ill-posed linear inverse problem is recast as the problem of minimizing a cost functional $\mathcal{J}:\Omega \to \R_+$, where $\Omega\subseteq \R^n$ and $\R_+$ is the set of non-negative real numbers:
\begin{equation}\label{eq:var_mod}
\bm{u}^*\;{\in}\;\arg\min_{\bm{u}\,\in\,\Omega}\left\{\mathcal{J}(\bm{u})\;{:=}\;\mathcal{R}(\bm{u})\,+\,\lambda\,\mathcal{F}(\bm{\mathrm{A}u};\bm{b})\right\} .
\end{equation}
Typically, in imaging problems subject to Poisson noise, $\Omega$ is the non-negative orthant of $\R^n$. 
The functional $\mathcal{R}:\Omega \to \R_+$ is the so-called \emph{regularization term}, encoding prior information on the unknown, while $\mathcal{F}:\Omega\to\R_+$ is the \emph{fidelity term}, which measures the distance between the data $\bm{b}$ and the blurred image $\bm{\mathrm{A}u}$ in a way that accounts for the noise statistics. Finally, the regularization parameter $\lambda>0$ balances the actions of the fidelity and regularization terms in the overall functional $\mathcal{J}$, and its setting is crucial in order to obtain high-quality restorations. 

In presence of Poisson noise, the fidelity term is typically set as the Kullback-Leibler divergence of $\bm{\mathrm{A}u}+\bm{\gamma}$ from $\bm{b}$, i.e.,
\begin{equation}
\label{eq:KL} 
\mathcal{F}(\bm{\mathrm{A}u};\bm{b})\;{=}\;  D_{KL}(\bm{\mathrm{A}u}+\bm{\gamma};\bm{b})\;{=}\;\sum_{i=1}^n F((\bm{\A u})_i+\gamma_i;\bm{b}_i)
\end{equation}
where
\begin{equation}
\label{eq:KL2} 
F((\bm{\A u})_i+\gamma_i;\bm{b}_i)\;{:=}\; (\,(\bm{\mathrm{A}u})_i\,{+}\,\gamma_i\,) - b_i \ln (\,(\bm{\mathrm{A}u})_i\,{+}\,\gamma_i\,) + b_i \ln b_i - b_i \,,
\end{equation}
and we assume $b_i\ln b_i=0$ for $b_i=0$.


The choice of a suitable regularization term is a very delicate issue, which is typically related to the nature of the processed image. A widespread regularizer is the popular Total Variation (TV) \cite{tvrof}, which reads
\begin{equation}\label{eq:tv}
\mathcal{R}(\bm{u})\;{=}\; \mathrm{TV}(\bm{u})\;{:=}\;\sum_{i=1}^n\|(\bm{\D u})_i\|_2\,,    
\end{equation}
where $\bm{\D}=(\bm{\D}_h,\bm{\D}_v)\in\R^{2n\times n}$ denotes the discrete gradient operator, with $\bm{\D}_h$, $\bm{\D}_v\in\R^{n\times n}$
two linear operators representing finite-difference discretizations of the first-order partial derivatives of the image $\bm{u}$ in the
horizontal and vertical directions. By plugging \eqref{eq:KL}, \eqref{eq:KL2} and \eqref{eq:tv} into \eqref{eq:var_mod}, we get what we will refer to as the TV-KL variational model:
\begin{equation}
\label{eq:tvKL}
\tag{TV-KL}
\bm{u}^*\;{\in}\;\arg\min_{\bm{u}\,\in\,\Omega}\left\{\mathrm{TV}(\bm{u})\,+\,\lambda\,D_{KL}(\bm{\mathrm{A}u};\bm{b})\right\}\,.
\end{equation}

A plethora of works has focused on the analytical study of the (non-smooth) TV term, which presents very desirable properties such as convexity and edge preservation. However, TV is also affected by some intrinsic limitations; above all, we mention the \emph{staircasing} effect, which can be thought of as the tendency of promoting edges at the expense of smooth structures.

Many attempts have been made in order to design regularization terms capable of overcoming the TV limitations. In this perspective, one of the most successful regularizers is the Total Generalized Variation (TGV) \cite{tgv,tgv2}, which in its second order discrete formulation is defined as
\begin{equation}\label{eq:tgv2}
\mathcal{R}(\bm{u})\;{=}\;  \mathrm{TGV}^2(\bm{u},\bm{w}) \;{:=}\; \min_{w\in\R^{2n}}\left\{\alpha_0 \sum_{i=1}^n\|(\bm{\D u})_i-\bm{w}_i\|_2 + \alpha_1\sum_{i=1}^n\|(\bm{\mathcal{E}w})_i\|_2\right\}\,,
\end{equation}
where $\bm{w}_i=((\bm{w}_1)_i,(\bm{w}_2)_i)\in\R^2$,
\begin{equation}
\bm{\mathcal{E}} \;{=}\; \begin{bmatrix}
\bm{\bm{\D}}_h & \bm{0}_{n\times n}\\
\frac{1}{2}\bm{\bm{\D}}_v &\frac{1}{2}\bm{\bm{\D}}_h\\
\frac{1}{2}\bm{\bm{\D}}_v &\frac{1}{2}\bm{\bm{\D}}_h\\
\bm{0}_{n\times n} &\bm{\bm{\D}}_v
\end{bmatrix}\,\in\R^{4n\times 2n}
\end{equation}
is the discrete symmetrized gradient operator, and $\alpha_0$, $\alpha_1$ are positive parameters.
By plugging \eqref{eq:KL}, \eqref{eq:KL2} and \eqref{eq:tgv2} into \eqref{eq:var_mod}, we obtain the TGV$^2$-KL model:
\begin{equation}
\label{eq:tgvKL}
\tag{TGV$^2$-KL}
\bm{u}^*\;{\in}\;\arg\min_{\bm{u}\,\in\,\Omega}\left\{\mathrm{TGV}^2(\bm{u})\,+\,\lambda\,D_{KL}(\bm{\mathrm{A}u};\bm{b})\right\}\,.
\end{equation}
The TGV$^2$ regularization has been proven to overcome the TV staircasing effect as it is also capable of promoting piece-wise affine structures. However, this capability is subject to a suitable selection of the parameters $\alpha_0$ and $\alpha_1$, which can highly influence the quality of the output restoration \cite{kostas}. 

In Figure~\ref{fig:teaser}, we show the results obtained by applying the TGV$^2$-KL model with $\lambda=1$ and different choices of $(\alpha_0,\alpha_1)$ to the restoration of the test image \texttt{penguin} corrupted by Gaussian blur and Poisson noise. Upon the selection of a large value of $\alpha_1$ (see Figure~\ref{fig:peng_1}), TGV$^2$ behaves as a TV regularizer and promotes piece-wise constant structures. On the other hand, when $\alpha_0$ is too large, as for the restoration in Figure~\ref{fig:peng_2}, the TGV$^2$ acts as a TV$^2$ regularization term and it is not capable of preserving edges. The two limit behaviors can be compared with the restoration shown in Figure~\ref{fig:peng_3} and corresponding to an `optimal' selection of $(\alpha_0,\alpha_1)$ in terms of quality measures. 

\begin{figure}
	\centering
	\begin{subfigure}[t]{0.4\textwidth}
		\centering
		\begin{overpic}[height=4.5cm]{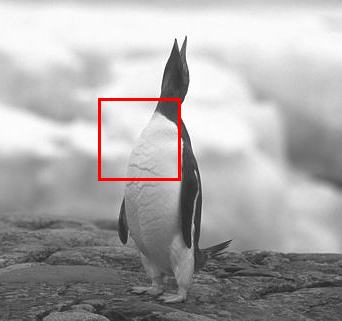}
			\put(56,.4){\color{red}%
				\frame{\includegraphics[height = 2.1cm]{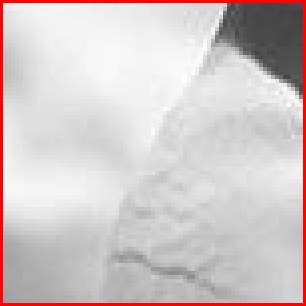}}}
		\end{overpic}    
		\caption{original}
		\label{fig:peng_isnr}
	\end{subfigure}
	\begin{subfigure}[t]{0.4\textwidth}
		\centering
		\begin{overpic}[height=4.5cm]{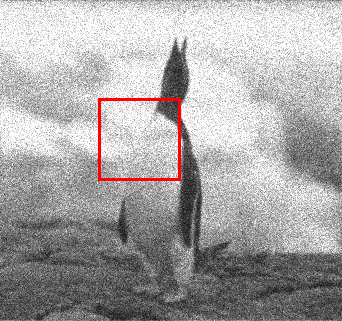}
			\put(56,.4){\color{red}%
				\frame{\includegraphics[height = 2.1cm]{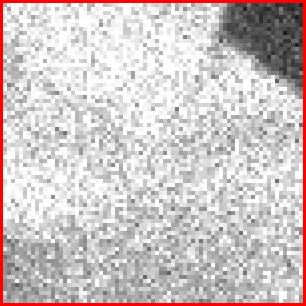}}}
		\end{overpic} 
		\caption{data}
		\label{fig:peng_ssim}
	\end{subfigure}\\
	\begin{subfigure}[t]{0.32\textwidth}
		\centering
		\begin{overpic}[height=4.5cm]{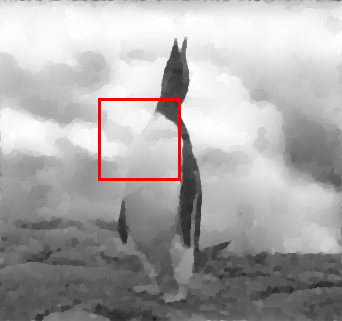}
			\put(56,.4){\color{red}%
				\frame{\includegraphics[height = 2.1cm]{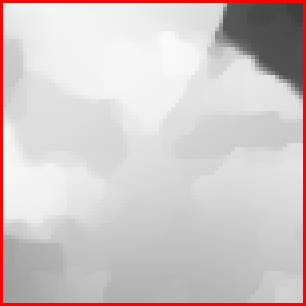}}}
		\end{overpic} \caption{$(\alpha_0,\alpha_1)=(0.1,0.6)$}
		\label{fig:peng_1}
	\end{subfigure}
	\begin{subfigure}[t]{0.32\textwidth}
		\centering
		\begin{overpic}[height=4.5cm]{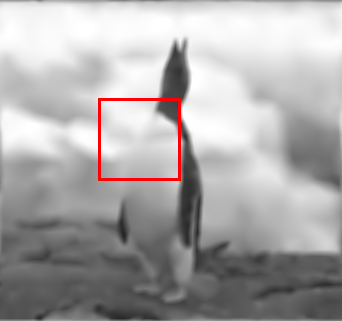}
			\put(56,.4){\color{red}%
				\frame{\includegraphics[height = 2.1cm]{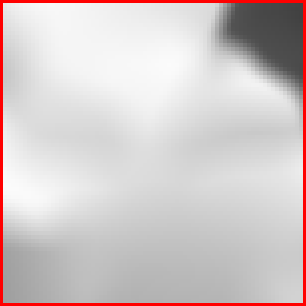}}}
		\end{overpic} \subcaption{$(\alpha_0,\alpha_1)=(0.6,0.3)$}
		\label{fig:peng_2}
	\end{subfigure}
	\begin{subfigure}[t]{0.32\textwidth}
		\centering
		\begin{overpic}[height=4.5cm]{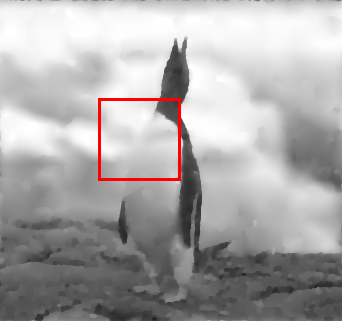}
			\put(56,.4){\color{red}%
				\frame{\includegraphics[height = 2.1cm]{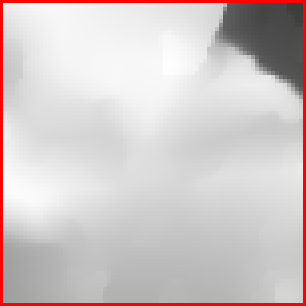}}}
		\end{overpic} \caption{$(\alpha_0,\alpha_1)=(0.1,0.3)$}
		\label{fig:peng_3}
	\end{subfigure}
	\caption{Original image (a), observed data corrupted by blur and Poisson noise (b), solution of the TGV$^2$-KL model for different selections of the parameters $(\alpha_0, \alpha_1)$ (c,d,e).\label{fig:teaser}}
\end{figure}

Besides some heuristic strategies - see, e.g., \cite{knoll,knoll2} - several works have been focused on designing a selection method which does not depend on the application at hand and on the class of images considered. In this direction, a bilevel optimization scheme has been designed in \cite{carola} for the automatic estimation of the TGV$^2$ parameters in image denoising problems, i.e., when $\bm{\A}=\bm{\I}_{n}$. More recently, the empirical Bayesian approach proposed in \cite{pereyra} for the automatic estimation of the regularization parameter in imaging problems has been applied for the selection of $(\alpha_0,\alpha_1)$. Therein, the TGV$^2$ prior has been slightly modified so as to allow the adoption of the proposed method, which ultimately relies on exploring the likelihood distribution via a Monte Carlo sampling technique, in image denoising problems.

\subsection{Contribution}

In this paper we present an automatic procedure for the solution of the TGV$^2$-KL model applied to the restoration of images undergoing the formation model in~\eqref{eq:lin_mod}. Our approach relies on the hierarchical Bayesian paradigm according to which the unknown image $\bm{u}$ as well as the unknown parameters $\alpha_0$ and $\alpha_1$ can be modeled as random variables. The regularization parameter $\lambda$, which does not naturally follows from the statistical derivation, is artificially included in the final joint model; this choice is motivated by previous results on fully hierarchical Bayesian approaches~\cite{bit}.

From the numerical viewpoint, the final joint model will be tackled via an alternating minimization scheme; the minimization problem with respect to $\bm{u}$ is performed via the the Alternating Direction Method of Multipliers (ADMM) equipped with an automatic procedure for the selection of $\lambda$. The step for updating the parameters is based on the current iterate and on a robust initialization, which ultimately facilitates the convergence of the overall scheme.

The proposed strategy returns high-quality restorations with image quality measures that are close to the best ones we achieved by manually tuning the parameters.



\medskip 

The paper is organized as follows. In Section~\ref{sec:not} we set the notations and report some preliminary results that will be useful in the next sections. The statistical derivation leading to the joint variational model is carried out in Section~\ref{sec:map}. In Section~\ref{sec:AM} we outline the alternating scheme and study analytical properties of the arising minimization problems, which are solved by means of the numerical method detailed in Section~\ref{sec:num}. In Section~\ref{sec:ex} we report the results of the proposed approach applied to some image restoration problems. Finally, we draw some conclusions and outline possible future research in Section~\ref{sec:concl}.

\section{Notations and preliminaries\label{sec:not}}

Throughout the paper, we denote by $\R_+$ the set of non-negative real numbers and by $\R_{++}$ the set $\R_+\setminus\{0\}$. We also denote by $\R_+^n$ and $\R_{++}^n$ the non-negative and the positive orthant of $\R^n$. We use the notations null($\bm{\mathrm{M}}$),  $\bm{0}_{n}$ and $\bm{\mathrm{I}}_{n}$ to denote the null space of the linear operator $\bm{\mathrm{M}}$, the vector with $n$ zero entries, and the $n\times n$ identity matrix, respectively. Given any vectors $\bm{v}$ and $\bm{w}$, for brevity we write $(\bm{v}, \bm{w})$ instead of $(\bm{v}^T, \bm{w}^T)^T$. Random variables are indicated by capital letters, e.g., $X$, with $x$ denoting a realization of $X$. The probability density function (pdf) of a continuous random variable $X$ is indicated by $\pi_X$, while the probability mass function (pmf) of a discrete random variable $X$ is indicated by $\mathrm{P}_X$. The characteristic function $\chi_S$ and the indicator function $\iota_S(x)$ of a set $S$ are defined as follows:
%
%
\begin{equation}
\chi_S(x) \;{=}\; \left\{
\begin{array}{cc}
1&\mathrm{if}\;x\,\in\,S   \\
0& \mathrm{otherwise}
\end{array}
\right.\,,\qquad 
\iota_S(x) \;{:=}\; - \ln \chi_{S}(x) \;{=}\;  \left\{
\begin{array}{cc}
0&\mathrm{if}\;x\,\in\,S   \\
+\infty& \mathrm{otherwise}
\end{array}
\right.\,.
\end{equation}

For completeness, we recall the definitions of few well-known distributions that are used in Section~\ref{sec:map}.

\begin{definition}[Poisson distribution]
	\label{def:pois}
	The discrete univariate random variable $X$ follows a Poisson distribution with mean parameter $\eta\in\R_{++}$, in short $X\sim \mathrm{Poiss}(\eta)$, if its pmf has the following form:
	\begin{equation}
	\mathrm{P}_{X}(x)\;{=}\;\frac{\eta^{x}\,e^{-\eta}}{x!}\,.
	\end{equation}
\end{definition}

\begin{definition}[Exponential distribution]
	The continuous univariate random variable $X$ follows an exponential distribution with scale parameter $\alpha\in\R_{++}$, in short $X\sim \mathrm{Exp}(\alpha)$, if its pdf has the following form:
	\begin{equation}
	\pi_{X}(x)\;{=}\;\alpha\,e^{-\alpha\,x}\,.
	\end{equation}
\end{definition}

\begin{definition}[Gamma distribution]\label{def:gam}
	The continuous univariate random variable $X$ follows a gamma distribution with scale parameter $\theta\in\R_{++}$ and shape parameter $k\in\R_{++}$, in short $X\sim \mathrm{Gamma}(k,\theta)$, if its pdf has the following form:
	\begin{equation}\label{eq:gam}
	\pi_{X}(x)\;{=}\;\frac{1}{\Gamma(k)\,\theta^k}\,x^{k-1}\,e^{-\frac{x}{\theta}}\, ,
	\end{equation}
	where $\Gamma(k)$ is the gamma function at $k$. The mode $m$ and the variance $\sigma^2$ of $X\sim\mathrm{Gamma}(k,\theta)$ are given by
	\begin{equation}\label{eq:defms}
	m = (k-1)\,\theta\,,\quad \sigma^2 = k\,\theta^2\,,
	\end{equation}
	with the mode $m$ being defined only for $k\geq 1$.
\end{definition}

In Figure~\ref{fig:dist}, we report the behavior of the exponential (left) and gamma (right) pdfs for different selections of the parameters $\alpha$ and $(k,\theta)$, respectively. We remark that the gamma pdf reduces to the exponential pdf when $k=1$.

\begin{figure}
	\centering
	\begin{subfigure}[t]{0.4\textwidth}
		\centering
		\includegraphics[width=5.5cm]{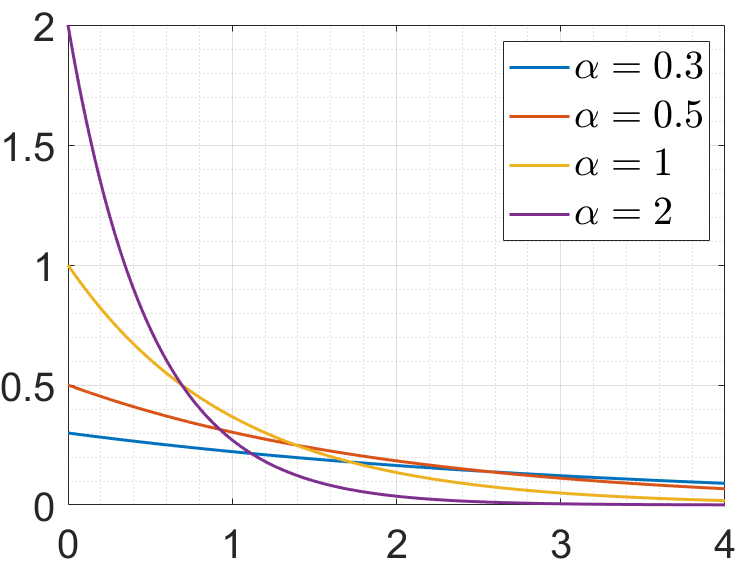}
		\caption{}
		\label{fig:exp}
	\end{subfigure}
	\begin{subfigure}[t]{0.4\textwidth}
		\centering
		\includegraphics[width=5.5cm]{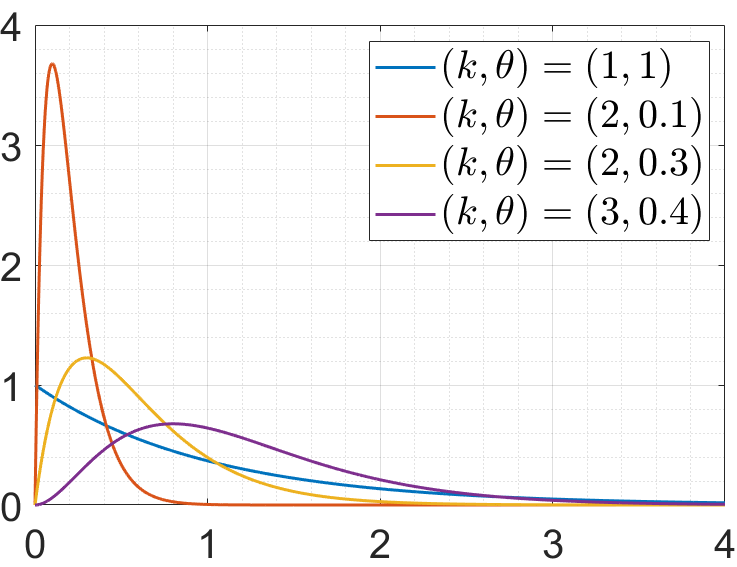}
		\caption{}
		\label{fig:gam}
	\end{subfigure}
	\caption{Probability density functions of the exponential (a) and gamma distribution (b) for different values of the parameters.}
	\label{fig:dist}
\end{figure}

\section{Statistical derivation\label{sec:map}}

Statistical approaches for image processing have
become very popular in the last decades due to their ability to incorporate non-deterministic and qualitative information in the forward model - see \cite{Calvetti2011}. According to the Bayesian paradigm, the unknowns of interest are modeled as random variables. Here, we adopt a hierarchical approach, which means that the unknown of primary interest, amounting to the image $\bm{u}$ and to the fictitious gradient image $\bm{w}$, is augmented with the unknown parameters $\alpha_0$ and $\alpha_1$~\cite{hyper}. The information or beliefs available on the unknowns are encoded in the so-called \emph{prior} pdf $\pi(\bm{u},\bm{w},\alpha_0,\alpha_1)$. In the same way, the observation $\bm{b}$ is regarded as a realization of a discrete random variable whose behavior, when the unknowns are fixed, is encoded in the \emph{likelihood} pmf $\mathrm{P}(\bm{b}\mid \bm{u},\bm{w},\alpha_0,\alpha_1)$. In this framework, the goal is to recover the \emph{posterior} pdf $\pi(\bm{u},\bm{w},\alpha_0,\alpha_1\mid\bm{b})$, which informs on the behavior of the unknowns once that a given observation has been acquired. The posterior pdf is related to the prior pdf and the likelihood pmf via the Bayes' formula 
\begin{align}\label{eq:bayes}
\begin{split}
\pi(\bm{u},\bm{w},\alpha_0,\alpha_1\mid\bm{b}) &\;{=}\; \frac{\mathrm{P}(\bm{b}\mid \bm{u},\bm{w},\alpha_0,\alpha_1)\,\pi(\bm{u},\bm{w},\alpha_0,\alpha_1)}{\mathrm{P}(\bm{b})} \\
&\;{\propto}\; \mathrm{P}(\bm{b}\mid \bm{u},\bm{w},\alpha_0,\alpha_1)\,\pi(\bm{u},\bm{w},\alpha_0,\alpha_1)\,,
\end{split}
\end{align}
where the evidence term $\mathrm{P}(b)$ can be neglected as it is constant with respect to the unknowns of interest. 

\medskip

We start detailing the expression of the likelihood pmf which, in light of the independence of the Poisson realizations $b_1,\ldots,b_n$ and based on Definition \ref{def:pois}, reads:
\begin{equation}\label{eq:lik}
\mathrm{P}(\bm{b}\mid\bm{u}) \;{=}\; \prod_{i=1}^{n}\mathrm{P}(b_i\mid u_i) \;{=}\; \prod_{i=1}^{n}\frac{ (\,(\bm{\mathrm{A}u})_i\,{+}\,\gamma_i\,)^{b_i}\,e^{-(\,(\bm{\mathrm{A}u})_i\,{+}\,\gamma_i\,)}}{b_i!}\,.
\end{equation}
Note that the likelihood pmf is actually conditioned on the sole $\bm{u}$. For what concerns the joint prior pdf on the unknown $(\bm{u},\bm{w},\alpha_0,\alpha_1)$, we have:
\begin{equation}
\pi(\bm{u},\bm{w},\alpha_0,\alpha_1) \;{=} \pi(\bm{u},\bm{w}\mid \alpha_0,\alpha_1)\,\pi(\alpha_0,\alpha_1)\,,
\end{equation}
where $\pi(\alpha_0,\alpha_1)$ is the \emph{hyperprior} that accounts for the lack of information on the parameters and incorporate prior information or beliefs on their behavior. For $\bm{u}$ not depending on $\alpha_1$ and $\bm{w}$ not depending on $\alpha_0$, the conditioned prior $\pi(\bm{u,\bm{w}}\mid \alpha_0,\alpha_1)$ is treated as follows:
\begin{equation}
\pi(\bm{u},\bm{w}\mid\alpha_0,\alpha_1) \;{=}\; \pi(\bm{u}\mid \alpha_0,\bm{w})\, \pi_{0}(\bm{w}\mid \alpha_1) \;{=}\;\underbrace{\pi_{c}(\bm{u})\,\pi_{0}(\bm{u}\mid \alpha_0,\bm{w})}_{\pi(\bm{u}\mid\alpha_0,\bm{w})}\,\pi_{0}(\bm{w}\mid \alpha_1)
\end{equation}
where $\pi_{c}(\bm{u}) := \chi_{\Omega}(\bm{u})$ is an \emph{improper} prior (i.e. it does not normalize to 1) accounting for the non-negativity of $\bm{u}$ \cite{ten}, 
while $\pi_0(\bm{u}\mid\alpha_0,\bm{w})$ and $\pi_0(\bm{w}\mid\alpha_1)$ encode the expected behavior of $\bm{u}$ and $\bm{w}$, respectively, which is actually expressed in terms of the expected behavior of 
\begin{eqnarray}\label{eq:zu}
& z_u:\R^n\to\R\,,\qquad &z_u(\bm{u}\,;\,\bm{w}) \;{=}\; \bm{\D u}-\bm{w} ,\\
\label{eq:zw}
& z_w:\R^{2n}\to \R\,,\qquad &z_w(\bm{w})\;{=}\; \bm{\mathcal{E}w} .
\end{eqnarray}
Here, we assume that the magnitudes of the entries of $z_u(\bm{u}\,;\,\bm{w})$, $z_w(\bm{w})$ are independent realizations of an exponential distribution with parameters $\alpha_0$, $\alpha_1$, respectively, that is
\begin{align}
\label{eq:pr01}
\pi_0(z_u(\bm{u}\,;\,\bm{w})\mid\alpha_0) \;{=}\;&     \alpha_0^n\, \exp\left(-\alpha_0 \sum_{i=1}^n\|(\bm{\D u})_i-\bm{w}_i\|_2\right)\,,\\
\label{eq:pr02}
\pi_0(z_w(\bm{w})\mid\alpha_1) \;{=}\; &\alpha_1^n\, \exp\left(-\alpha_1 \sum_{i=1}^n\|(\bm{\mathcal{E}w})_i\|_2\right)\,.
\end{align}
We thus have
\begin{equation}\label{eq:pr}
\pi_0(\bm{u}\mid \alpha_0,\bm{w}) \;{=}\; \frac{1}{Z(\alpha_0)}\,\pi_0(z_u(\bm{u}\,;\,\bm{w})\mid\alpha_0)\,,\quad 
\pi_0(\bm{w}\mid \alpha_1) \;{=}\;\frac{1}{Z(\alpha_1)} \,\pi_0(z_w(\bm{w})\mid\alpha_1)\,,
\end{equation}
where $Z(\alpha_0)$ and $Z(\alpha_1)$ are referred to as partition functions and guarantee that the above prior pdfs normalize to $1$ when considering all the possible configurations of $\bm{u}$ and $\bm{w}$, respectively. In general, a closed form expression for the partition function for TV-type priors is not available, and the same holds for the pdfs in \eqref{eq:pr}. As a consequence, one can either approximate it - as in \cite{part1,part2} - or neglect it - see \cite{part3}. We remark that the approximation of the partition function is typically performed relying on the erroneous assumption of independence of the entries of $z_u(\bm{u}\,;\,\bm{w})$, $z_w(\bm{w})$, which clearly introduce a certain amount of inexactness in the resulting prior pdf. In what follows, we are going to assume $Z(\alpha_0),Z(\alpha_1)$ to be constant with respect to $\alpha_0,\alpha_1$, respectively, so that they can be neglected. At the end of this section, we discuss the consequences of this choice. We also remark that our selection of priors \eqref{eq:pr01}, \eqref{eq:pr02} is tantamount to assume that both $\bm{u}$, $\bm{w}$ are modeled as Markov Random Fields with the associate Gibbs' priors given by \eqref{eq:pr01}, \eqref{eq:pr02} and \eqref{eq:pr} \cite{geman}.

Based on the Bayes' formula in \eqref{eq:bayes} and on our selection of likelihood pmf and prior pdf, we have that the posterior pdf can be obtained as follows:
\begin{equation}\label{eq:post}
\pi(\bm{u},\bm{w},\alpha_0,\alpha_1\mid \bm{b}) \propto \mathrm{P}(\bm{b}\mid\bm{u})\,\pi_c(\bm{u})\,\pi_0(\bm{u}\mid \alpha_0,\bm{w})\, \pi_0(\bm{w}\mid \alpha_1)\, \pi(\alpha_0,\alpha_1)\,.
\end{equation}

In the context of hierarchical Bayesian approaches, the unknown parameters arising in the posterior pdf can either be removed by marginalization or can be estimated jointly with the unknowns $\bm{u},\bm{w}$ \cite{marg}. Here, we resort to the latter approach and invoke the popuplar Maximum A Posteriori (MAP) estimation strategy to synthesize the posterior pdf with its mode:
\begin{equation}
\{\bm{u}^*,\bm{w}^*,\alpha_0^*,\alpha_1^*\}\;{\in}\;\argmax_{\bm{u},\bm{w},\alpha_0,\alpha_1} \left\{
\mathrm{P}(\bm{b}\mid\bm{u})\,\pi_c(\bm{u})\,\pi_0(\bm{u}\mid \alpha_0,\bm{w})\, \pi_0(\bm{w}\mid \alpha_1)\, \pi(\alpha_0,\alpha_1)\right\}\,,
\end{equation}
or, equivalently,
\begin{align}\label{eq:map}
\begin{split}
\{\bm{u}^*,\bm{w}^*,\alpha_0^*,\alpha_1^*\}\;{\in}\;\arg\min_{\bm{u},\bm{w},\alpha_0,\alpha_1} \big\{
&-\ln\mathrm{P}(\bm{b}\mid\bm{u})-\ln\pi_c(\bm{u})\,-\ln\pi_0(\bm{u}\mid \alpha_0,\bm{w})\\
&-\ln\pi_0(\bm{w}\mid \alpha_1)-\ln \pi(\alpha_0,\alpha_1)\big\}\,.
\end{split}
\end{align}
By plugging \eqref{eq:lik} and \eqref{eq:pr} into \eqref{eq:map}, and by properly adjusting the constant terms, we get
\begin{equation}\label{eq:map1}
\{\bm{u}^*,\bm{w}^*,\alpha_0^*,\alpha_1^*\}\;{\in}\;\arg\min_{(\bm{u},\bm{w},\alpha_0,\alpha_1)\,\in\, \mathcal{C}} \bigg\{D_{KL}(\bm{\A u}{+}\bm{\gamma};\bm{b}){+}\mathcal{R}(\bm{u},\bm{w},\alpha_0,\alpha_1)
{+}\mathcal{P}(\alpha_0,\alpha_1)\bigg\}\,,
\end{equation}
where $\mathcal{C}=\Omega \times \R^{2n}\times\R^2_{++}$ and
\begin{equation}\label{eq:jreg}
\mathcal{R}(\bm{u},\bm{w},\alpha_0,\alpha_1) \;{:=}\;\alpha_0\sum_{i=1}^n\|(\bm{\D u})_i-\bm{w}_i\|_2+\alpha_1\sum_{i=1}^n\|(\bm{\mathcal{E}w})_i\|_2\,,
\end{equation}
is the joint non-differentiable regularization term, while
\begin{equation}\label{eq:pen}
\mathcal{P}(\alpha_0,\alpha_1) \;{:=}\; 
-n\ln\alpha_0-n\ln\alpha_1-\ln\pi(\alpha_0,\alpha_1)
\end{equation}
is the parameter penalty term.

The synergies between the hierarchical formulation and the MAP estimate have been explored in detail in the last few years in the context of sparse recovery problems \cite{Calvetti_2019,Calvetti_2020}. Nonetheless, it must be remarked that the MAP problem derived in the mentioned works is somehow \emph{exact}, as the the distribution considered therein for the primary unknown is equipped with a partition function that admits a closed form with respect to the parameters. 

\medskip

In what follows, we consider different scenarios for setting the hyperprior $\pi(\alpha_0,\alpha_1)$ under the assumption of independence of $\alpha_0,\alpha_1$, i.e. $\pi(\alpha_0,\alpha_1) \;{=}\; \pi(\alpha_0)\,\pi(\alpha_1)\,.$


\subsection{Noninformative hyperpriors}

When no prior information is available on the unknown parameters $\alpha_0$, $\alpha_1$, one can typically set \emph{noninformative} hyperpriors, which are somehow considered a \emph{default} choice in this scenario \cite{bayesian}.

Typical noninformative hyperpriors that have been employed for the parameter estimation task in imaging problems are the flat hyperprior (see, e.g., \cite{part2}), namely
\begin{equation}\label{eq:non1}
\pi(\alpha_0)\;{=}\;\pi(\alpha_1)\;{=}\;\rho\,\in\,\R_{++} \implies \mathcal{P}(\alpha_0,\alpha_1) = -n\ln\alpha_0-n\ln\alpha_1\,,
\end{equation}
and the Jeffrey's hyperpriors (see, e.g., \cite{jeff}), which read
\begin{equation}\label{eq:non2}
\pi(\alpha_0)\;{\propto}\;\frac{1}{\alpha_0}\,,\, \pi(\alpha_1)\;{\propto}\;\frac{1}{\alpha_1}\implies \mathcal{P}(\alpha_0,\alpha_1) = -(n-1)\ln\alpha_0-(n-1)\ln\alpha_1\,.
\end{equation}
One can easily notice that the above hyperpriors are both improper, as they do not normalize to 1. Nonetheless, improper priors and hyperpriors are often employed in practice as they are still capable of encoding qualitative information on the unknown of interest \cite{bayesian}.

Generally speaking, the major risk that may occur when adopting a noninformative hyperprior is that the final posterior distribution may result to be particularly sensitive to small changes in the parametrized prior distribution.
The mentioned lack of robustness has an analytical counterpart which comes from observing that, upon the selection of the noninformative hyperpriors considered above, the term $\mathcal{P}(\alpha_0,\alpha_1)$ approaches $-\infty$ as $\alpha_0$ or $\alpha_1$ tends to $+\infty$ - see \eqref{eq:non1} and \eqref{eq:non2}. As a result, when either $\bm{\D u}=\bm{w}$ or $\bm{w}$ is a constant image, that is one of the two summations in \eqref{eq:jreg} vanishes, the overall cost functional in \eqref{eq:map1} is unbounded from below and does not attain a minimum.

\subsection{Informative hyperpriors}

The potential instability related to the selection of noninformative hyperpriors can be circumvented by exploiting the (possibly poor) information available on the unknown parameters.

Let us assume that we can approximately determine the interval of $\R_+$ which is most likely to include the sought values $\alpha_0^*$ and $\alpha_1^*$. In this scenario, one can select a hyperprior that
\begin{itemize}
	\item[(a)] is capable of driving the search of $\alpha_0$ and $\alpha_1$ in the detected interval;
	\item[(b)] allows for outliers in case the information about the interval is too poor;
	\item[(c)] makes the resulting parameter penalty term easily tractable.
\end{itemize}
In this perspective, a hyperprior that is typically adopted to model the behavior of the unknown parameters in the hierarchical Bayesian literature is the gamma distribution (see Definition \ref{def:gam}) \cite{Calvetti_2019,part1,part2}. In fact, from Figure \ref{fig:gam}, it is clear that, depending on the values of the shape and scale parameters, the realizations of a gamma random variable are more likely to belong to a certain interval of $\R_+$ corresponding to the larger values assumed by the pdf. However, as highlighted in \cite{Calvetti_2019} in the context of sparse recovery problems, the heavy tail structure of the gamma distribution guarantees the realization of outliers, that is of values that can be very far from the mode of the distribution. Finally, as it will be remarked in Sections \ref{sec:AM} and \ref{sec:num}, the choice of a gamma hyperprior appears to be very convenient in terms of computational efficiency and analytical properties of $\mathcal{P}(\alpha_0,\alpha_1)$.

Hence, we assume that $\alpha_0$ and $\alpha_1$ follow a gamma distribution with parameters $(\theta_0,k_0)$ and $(\theta_1,k_1)$, respectively, with $\theta_0,\theta_1>0$ and $k_0,k_1>1$, the latter condition guaranteeing the existence of a finite mode. In Section \ref{sec:num}, we provide some details about setting the shape and scale parameters of the newly introduced gamma hyperpriors in a robust manner.


Recalling the expression of the gamma pdf given in \eqref{eq:gam}, by plugging it into \eqref{eq:pen} - keeping in mind that $\ln\pi(\alpha_0,\alpha_1)=\ln\pi(\alpha_0)+\ln\pi(\alpha_1)$ - and dropping out the constant terms, we get 
\begin{equation}\label{eq:pen2}
\mathcal{P}(\alpha_0,\alpha_1) = \,{-}\,(n+k_0-1)\ln\alpha_0\,{+}\,\frac{\alpha_0}{\theta_0}\,{-}\,(n+k_1-1)\ln\alpha_1\,{+}\,\frac{\alpha_1}{\theta_1}\,.
\end{equation}


\medskip

Before detailing the analytical properties of the cost functional in \eqref{eq:map1}, with $\mathcal{R}$ as in \eqref{eq:jreg} and $\mathcal{P}$ as in \eqref{eq:pen2}, we highlight that neglecting the partition functions $Z(\alpha_0)$, $Z(\alpha_1)$ arising in \eqref{eq:pr} simplifies the analytic expression of the MAP problem but, at the same time, introduces a certain degree of inexactness in the final variational model. In~\cite{bit}, for a different class of models, it has been already observed that the considered simplification can be responsible for a fragile behavior of the MAP estimate, which may return particularly oversmoothed restorations under the influence of the parameters $k_0,\theta_0,k_1,\theta_1$. To increase the robustness of the proposed strategy, without giving up the computational efficiency guaranteed by the simplified MAP approach, we artificially introduce a regularization parameter $\lambda$ multiplying the KL fidelity term. As a result, from now on, the problem of interest can be formulated as
\begin{equation}
\label{eq:map31}
\{\bm{u}^*,\bm{w}^*,\alpha_0^*,\alpha_1^*\}\;{\in}\;\arg\min_{(\bm{u,\bm{w},\alpha_0,\alpha_1})\,\in\,\mathcal{C}} \mathcal{J}(\bm{u},\bm{w},\alpha_0,\alpha_1)\,.
\end{equation}
with the cost functional $\mathcal{J}$ being written in extended form as
\begin{align}\label{eq:Jnew}
\begin{split}
\mathcal{J}(\bm{u},\bm{w},\alpha_0,\alpha_1)\;{:=}\;&\lambda\,D_{KL}(\bm{\A u}+\bm{\gamma};\bm{b})+\alpha_0\,\sum_{i=1}^n\|(\bm{\D u})_i-\bm{w}_i\|_2+\alpha_1\,\sum_{i=1}^n\|(\bm{\mathcal{E} w})_i\|_2\\
&\,{-}\,(n+k_0-1)\ln\alpha_0\,{+}\,\frac{\alpha_0}{\theta_0}\,{-}\,(n+k_1-1)\ln\alpha_1\,{+}\,\frac{\alpha_1}{\theta_1}\,.
\end{split}
\end{align}
Clearly, the above modification, which is expected to mitigate the oversmoothing effect by correctly balancing the contribution of the fidelity and of the penalty terms in the overall functional, has to be equipped with an automatic estimation strategy, which is discussed in the following section.

\section{Existence of solutions\label{sec:AM}}

In this section, we discuss analytical properties of the functional $\mathcal{J}$ in~\eqref{eq:Jnew} and introduce the alternating minimization scheme that will be employed for the solution of problem~\eqref{eq:map31}-\eqref{eq:Jnew}.

\medskip



We start recalling well-known properties of the KL fidelity term in~\eqref{eq:Jnew}, which have been studied in detail in~\cite{Bertero2010}. More specifically, the $D_{KL}(\cdot\,;\,\bm{b})$ functional is non-negative, coercive and convex on the non-negative orthant $\Omega := \R^n_+$. 
For what concerns the joint regularization term $\mathcal{R}$ in \eqref{eq:jreg}, one can easily observe that it is coercive in the joint variable $(\alpha_0,\alpha_1,\bm{u},\bm{w})$.
Finally, the penalty term $ \mathcal{P}$ yielded by the two Gamma hyperpriors is coercive and strictly convex; in fact, one can easily check that for each $(\alpha_0,\alpha_1)\in\R^2$ the Hessian matrix of $\mathcal{P}$ is positive-definite.

Based on the aforementioned properties, the following result on the existence of solutions for problem~\eqref{eq:map31}-\eqref{eq:Jnew} holds true.

\begin{proposition}[Existence of solutions]
	The functional $\mathcal{J}$ defined in \eqref{eq:Jnew} is lower semi-continuous and coercive on the set $\mathcal{C}=\Omega\times \R^{2n}\times \R_{++}^2$, and the minimization problem \eqref{eq:map31}-\eqref{eq:Jnew} admits global minimizers.
\end{proposition}

The minimization problem in \eqref{eq:map31}-\eqref{eq:Jnew} can thus be addressed by means of a block-coordinate descent method with respect to $(\bm{u,\bm{w}})$ and to $(\alpha_0,\alpha_1)$. Upon a suitable initialization $\alpha_0^{(0)}$, $\alpha_1^{(0)}$, and based on the definitions given in \eqref{eq:jreg}-\eqref{eq:pen}, the $k$-th iteration of the alternating scheme reads:
\begin{align}\label{eq:subuw}
\left\{\bm{u}^{(k+1)},\bm{w}^{(k+1)}\right\}\;{\in}\;&\arg\min_{(\bm{u},\bm{w})\,\in\,\Omega\times\R^{2n}}\left\{\lambda\,D_{KL}(\bm{\A u}+\bm{\gamma};\bm{b})\;{+}\;\mathcal{R}(\bm{u},\bm{w},\alpha_0^{(k)},\alpha_1^{(k)})\right\}\\
\label{eq:subal}
\left\{\alpha_0^{(k+1)},\alpha_1^{(k+1)}\right\}\;{\in}\;&\arg\min_{(\alpha_0,\alpha_1)\,\in\,\R_{++}^2}\left\{\mathcal{R}(\bm{u}^{(k+1)},\bm{w}^{(k+1)},\alpha_0,\alpha_1)\;{+}\;\mathcal{P}(\alpha_0,\alpha_1)\right\}\,.
\end{align}
%
%
First, we observe that problem \eqref{eq:subuw} can be equivalently written as follows
\begin{equation}\label{eq:subuw2}
\bm{u}^{(k+1)}\;{\in}\;\arg\min_{\bm{u}\,\in\,\Omega}\left\{\lambda\,D_{KL}(\bm{\A u}+\bm{\gamma};\bm{b})\;{+}\;\mathrm{TGV}^2(\bm{u})\right\}\,,
\end{equation}
with TGV$^2(\bm{u})$ being as in \eqref{eq:tgv2} with parameters $\alpha_{0}^{(k)}$ and $\alpha_{1}^{(k)}$.

Hence, based on the TGV$^2(\bm{u})$ properties \cite{tgv,tgv2} and on the analysis carried out at the beginning of this section, we have that the following results on the existence of minimizers for the problems in \eqref{eq:subal} and \eqref{eq:subuw2} hold true.

\begin{proposition}
	The cost function in \eqref{eq:subuw2} is convex, lower semi-continuous and coercive on $\Omega\times\R^{2n}$, hence it admits global minimizers.
\end{proposition}

\begin{proposition}
	The cost function in \eqref{eq:subal} is strictly convex on $\R_{++}^{2}$, hence it admits a unique global minimizer.
\end{proposition}



As our purpose is to design an automatic strategy, we also need to fix a criterion for the automatic selection of $\lambda$ in \eqref{eq:subuw2}. In the next section, we recall a popular strategy for the $\lambda$ estimation problem in presence of Poisson noise.

\subsection{Discrepancy principle under Poisson noise corruption}


Consider the variational model
\begin{equation} \label{eq:var_pois}
\bm{u}^*(\lambda) \in\arg\min_{\bm{u}\in\Omega}\left\{\lambda\,D_{KL}(\bm{\A u}+\bm{\gamma};\bm{b})+\mathcal{R}(\bm{u})\right\}\,,
\end{equation}
where in \eqref{eq:var_pois} we made explicit the dependence of the solution $\bm{u}^*$ on the parameter $\lambda$.

The extension of the Morozov Discrepancy Principle (DP), originally introduced for Gaussian noise corruption, to the Poisson degradation case reads:

\begin{equation}
\mathrm{select}\;\:\, \lambda \;{=}\; \lambda^* \in \R_{++} \;\;\, \mathrm{s.t.}\;\, 
\mathcal{D}\left(\lambda^*\,;\,\bm{b}\right) \;{=}\; \Delta  \, ,
\label{eq:DP0}
\end{equation}
where the last equality and the scalar $\Delta \in\R_{++}$ are commonly referred to as the \emph{discrepancy equation} and the \emph{discrepancy value}, respectively, while the \emph{discrepancy function} $\,\mathcal{D}(\,\cdot\,;\bm{b}): {\R_{++}} \to \R_+$ is defined as
%


\begin{equation}
\mathcal{D}\left(\lambda\,;\,\bm{b}\right)
\,\;{:=}\;\,
\sum_{i=1}^n \mathcal{D}_i\left(\lambda\,;\,b_i\right) , 
\label{eq:DP1}
\end{equation}
%
where
\begin{align}\label{eq:def_D}
\begin{split}
\mathcal{D}_i\left(\lambda\,;\,b_i\right) \;{:=}\;& F((\bm{\A u}^*(\lambda))_i+\gamma_i;b_i)\\
\;{=}\;&(\bm{\A u}^*(\lambda)+\bm{\gamma})_i - b_i\ln (\bm{\A u}^*(\lambda)+\bm{\gamma})_i +b_i\ln b_i - b_i\,,
\end{split}
\end{align}
with $F$ defined in \eqref{eq:KL2}.
In analogy to the Morozov DP, the discrepancy value $\Delta$ can be selected as the expected value of the KL fidelity term:
\begin{equation}
\mathrm{select}\;\:\, \lambda \;{=}\; \lambda^* \in \R_{++} \;\;\, \mathrm{s.t.}\;\,
\sum_{i=1}^n\mathcal{D}_i\left(\lambda^*;b_i\right) \;{=}\;\sum_{i=1}^n\mathbb{E}[F\left((\bm{\A u}^*(\lambda))_i+\gamma_i;b_i\right)]  \, .
\label{eq:DP1}
\end{equation}
However, getting an explicit expression for the expected values in \eqref{eq:DP1} is a very hard task. In \cite{Bertero2009,Bertero2010}, the authors consider the approximation
\begin{equation}\label{eq:approx}
\mathbb{E}[F\left((\bm{\A u}^*(\lambda))_i+\gamma_i;b_i\right)] \approx \frac{1}{2}\,,
\end{equation}
which comes from truncating the series expansion of $\mathbb{E}[F\left((\bm{\A u}^*(\lambda))_i+\gamma_i;b_i\right)]$ and yields the following formulation for the general principle outlined in \eqref{eq:DP0}, to which we will refer as Approximate DP (ADP):
\begin{equation}\tag{ADP}
\mathrm{select}\;\:\, \lambda \;{=}\; \lambda^* \in \R_{++} \;\;\, \mathrm{s.t.}\;\, 
\sum_{i=1}^n\mathcal{D}_i\left(\lambda^*;b_i\right) \;{=}\;\frac{n}{2}\,.
\label{eq:DP2}
\end{equation}  
For a specific, but not very restrictive, class of regularization terms,
it is possible to prove that there exists a
unique value $\lambda^*$ such that $\bm{u}^*(\lambda)$ satisfies the ADP - see \cite{Teuber,Benf_jmiv}. 
%



From the algorithmic viewpoint, the optimal $\lambda$ value can be sought in an \emph{a posteriori} manner, that is by computing the solution $\bm{u}^*(\lambda)$ of the general model \eqref{eq:var_mod} for different values of $\lambda$ and then selecting the one satisfying the constraint in \eqref{eq:DP2}. In alternative, when a closed form solution for \eqref{eq:var_pois} is not available, one can rather embed the ADP along the iterations of the alternating scheme employed for the solution of \eqref{eq:var_pois} \cite{Teuber,BF}. In the latter scenario, the solution of only one minimization problem is required.

Besides the desirable theoretical guarantees and the computational efficiency, it is well established that the ADP is not capable of achieving high-quality restorations in low-count regimes, that is when the number of photons allowed to hit the image domain is very low. This can be ascribed to the approximation in \eqref{eq:approx}, which is particularly rough in presence of few photons. Several strategies aimed at refining the estimate in \eqref{eq:approx} have been explored \cite{Bonettini_2014,bevi}. Despite the more robust behavior of these approaches in low-count regimes, they do not benefit from the theoretical guarantees that ADP brings along. More specifically, it is not trivial, or even not possible, to prove the uniqueness of the value of $\lambda$ satisfying the general instance of the DP in \eqref{eq:DP0}.

As the focus of our work is on the estimation of the parameters in the TGV$^2$ term, with the $\lambda$ estimation problem being a side effect, we will adopt the ADP in its iterated version. 
However, we remark that the numerical method detailed in Section \ref{sec:num} can be thought of as a machinery in which the $\lambda$-update strategy can be replaced by any of the other mentioned approaches.

\bigskip

\section{Numerical method\label{sec:num}}

In this section, we present the numerical procedure adopted for the solution of \eqref{eq:subuw}-\eqref{eq:subal}. First, we provide some details on the initialization strategy that allows us to compute the values $\alpha_0^{(0)}$ and $\alpha_1^{(0)}$ and to set the parameters of the gamma hyperpriors that arise in the penalty term $\mathcal{P}$ in \eqref{eq:pen2}. Then, we tackle the update of the variables of primary interest, i.e. $\bm{u}$ and $\bm{w}$ in \eqref{eq:subuw}, by means of an ADMM-based scheme where the regularization parameter $\lambda$ is automatically adjusted along the iterations according to the ADP. Finally, we detail how to perform the update of the unknown parameters $\alpha_0$ and $\alpha_1$ in \eqref{eq:subal}.

\subsection{Initialization of the alternating scheme\label{sec:init}}

We start by computing a suitable initial guess for the alternating scheme by solving the TV-KL model in \eqref{eq:tvKL} where the regularization parameter $\lambda$ is automatically selected according to the Discrepancy Principle described above \cite{Teuber}. 

The TV-regularized solution, which has been already adopted for the same purpose in image denoising problems~\cite{carola}, produces a twofold benefit: on one hand, it is expected to fasten the convergence of the alternating method \eqref{eq:subuw}-\eqref{eq:subal}, on the other, it allows us to reasonably set the parameters $k_0,\theta_0,k_1,\theta_1$ and the initial values $\alpha_0^{(0)}$, $\alpha_1^{(0)}$.


Denoting by $\bm{u}^{(0)}$ the solution of the TV-KL model coupled with the ADP criterion and setting $\bm{w}^{(0)}=\bm{\D u^{(0)}}$, we apply the Maximum Likelihood (ML) strategy for the estimation of parameters of an exponential distribution on the set of samples $z_u(\bm{b};\bm{w}^{(0)})$ and $z_w(\bm{w}^{(0)})$ defined in \eqref{eq:zu} and \eqref{eq:zw}, respectively. In other words, we select $\alpha_0^{(0)}$ and $\alpha_1^{(0)}$ by solving
\begin{equation}
\alpha_0^{(0)} = \argmax_{\alpha_0\,\in\,\R_{++}} \pi_0(z_u(\bm{b};\bm{w}^{(0)})\mid \alpha_0)\,,\quad \alpha_1^{(0)} = \argmax_{\alpha_1\,\in\,\R_{++}} \pi_0(z_w(\bm{w}^{(0)})\mid\alpha_1)\,.
\end{equation}
After turning the above maximization problems into minimization ones by taking the negative logarithm of the likelihood cost functions, and imposing a first-order optimality condition, we get
\begin{equation}
\alpha_0^{(0)} \;{=}\;  \left(\frac{1}{n}\sum_{i=1}^n\|\bm{w}_i^{(0)}-(\bm{\D b})_i\|_2\right)^{-1}\,,\quad
\alpha_1^{(0)} \;{=}\;  \left(\frac{1}{n}\sum_{i=1}^n\|(\bm{\mathcal{E} w}^{(0)})_i\|_2\right)^{-1}\,.
\end{equation}
The parameters $k_0,\theta_0,k_1,\theta_1$ of the gamma hyperpriors are fixed by selecting $\alpha_0^{(0)}$ and $\alpha_1^{(0)}$ as the modes of the distributions and by setting the standard deviations equal to $10^{-3}$. In other words, recalling the definitions given in \eqref{eq:defms}, we have
\begin{equation}
\theta_j(k_j-1) = \alpha_j^{(0)}\,,\quad \sqrt{k_j}\,\theta_j = 10^{-3}\,,\quad j=0,1\,.    
\end{equation}

\subsection{Updating $\bm{u}$ and $\bm{w}$\label{sec:admm}}

Now we consider the problem 
\begin{align}
\label{eq:subuwunc}
\begin{split}
\left\{\widehat{\bm{u}},\widehat{\bm{w}}\right\}\;{\in}\;&\arg\min_{\bm{u},\bm{w}}\big\{\lambda\,D_{KL}(\bm{\mathrm{A}u}+\bm{\gamma};\bm{b}) \,+\,\alpha_0\,\sum_{i=1}^n\|(\bm{\D u})_i-\bm{w}_i\|_2\\
&\phantom{XXXX}\,+\,\alpha_1\,\sum_{i=1}^n\|(\bm{\mathcal{E} w})_i\|_2\,+\,\iota_{\Omega}(\bm{u})\big\}\, ,
\end{split}
\end{align}
where we dropped out the superscript indices indicating the alternating scheme iteration for the sake of better readability, and where the parameters $\alpha_0,\alpha_1$ are fixed, while the regularization parameter is adjusted along the iterations of the iterative scheme so as to satisfy the ADP.


Problem~\eqref{eq:subuwunc} can be equivalently written in the following constrained form:
\begin{align}\label{eq:TGVKLcon}
\begin{split}
\{\widehat{\bm{u}},\widehat{\bm{w}}\}\;{\in}\;\arg\min_{\bm{u},\bm{w},\bm{z}_1,\bm{z}_2,\bm{z}_3,\bm{z}_4}\bigg\{&
\lambda D_{KL}(\bm{z}_1+\bm{\gamma};\bm{b})+\alpha_0\sum_{i=1}^n\|(\bm{z}_2)_i\|_2\\
&+\alpha_1\sum_{i=1}^n\|(\bm{z}_3)_i\|_2+\iota_{\Omega}(\bm{z}_4)
\bigg\}\\
\text{s.t.}\;\; \bm{z}_1\;{=}\; &\bm{\A u}\\
\phantom{\text{s.t.}\;\; }\bm{z}_2\;{=}\; &\bm{\D u}-\bm{w}\\
\phantom{\text{s.t.}\;\; }\bm{z}_3\;{=}\; &\bm{\mathcal{E}w }\\
\phantom{\text{s.t.}\;\; }\bm{z}_4\;{=}\; &\bm{u}\,,
\end{split}
\end{align}
where $\bm{z}_1,\bm{z}_4\in \R^n$, $\bm{z}_2\in\R^{2n}$, $\bm{z}_3\in\R^{4n}$ are auxiliary variables, $(\bm{z}_2)_i = ((\bm{\D}_h\bm{u})_i,(\bm{\D}_v\bm{u})_i)$ and $(\bm{z}_3)_i = ((\bm{\D}_h\bm{u})_i,(1/2)((\bm{\D}_h\bm{u})_i+(\bm{\D}_v\bm{u})_i),(1/2)((\bm{\D}_h\bm{u})_i+(\bm{\D}_v\bm{u})_i),(\bm{\D}_v\bm{u})_i)$.

By following the computations carried out in \cite{DLV_dtgv}, briefly reported here for the sake of completeness, and after introducing the variables $\bm{x}:=(\bm{u},\bm{w})$ and $\bm{z}:=(\bm{z}_1,\bm{z}_2,\bm{z}_3,\bm{z}_4)$ and the matrices
\begin{equation}
\bm{\mathrm{H}}=\left[
\begin{array}{cc}
\bm{\A} & \bm{0} \\
\bm{\D} & -\bm{\I}_{2n}\\
\bm{0}&\bm{\mathcal{E}}\\
\bm{\I}_n&\bm{0}
\end{array}
\right]\,,\quad   \bm{\G}=\left[
\begin{array}{cccc}
-\bm{\I}_n & \bm{0}& \bm{0}& \bm{0} \\
\bm{0} & -\bm{\I}_{2n} & \bm{0}& \bm{0}\\
\bm{0}& \bm{0}&-\bm{\I}_{4n}& \bm{0}\\
\bm{0}& \bm{0}& \bm{0}&-\bm{\I}_n
\end{array}
\right]\,,
\end{equation}
problem~\eqref{eq:TGVKLcon} can be reformulated as follows:
\begin{align}\label{eq:modelnew}
\begin{split}
\left\{\widehat{\bm{x}},\widehat{\bm{z}}\right\}\in&\arg\min_{\bm{x},\bm{z}}\left\{F_1(\bm{x})+F_2(\bm{z})\right\}\\
&\text{s.t.}\; \;\bm{\mathrm{H}x}+\bm{\G z}=0\,,
\end{split}
\end{align}
where
\begin{align}
F_1(\bm{x})\;{=}\;& 0,\\
F_2(\bm{z})\;{=}\;&\lambda D_{KL}(\bm{z}_1)+\alpha_0\sum_{i=1}^n\|(\bm{z}_2)_i\|_2+\alpha_1\sum_{i=1}^n\|(\bm{z}_3)_i\|_2+\iota_{\Omega}(\bm{z}_4)\,.
\end{align}
The augmented Lagrangian associated to problem \eqref{eq:modelnew} reads
\begin{equation}\label{eq:AL}
\mathcal{L}(\bm{x},\bm{z},\bm{\zeta};\rho) = F_1(\bm{x})+F_2(\bm{z})+\langle\bm{\zeta},\bm{\mathrm{H}x}+\bm{\G z}\rangle+\frac{\rho}{2}\|\bm{\mathrm{H}x}+\bm{\G z}\|_2^2\,,
\end{equation}
where $\bm{\zeta}=(\bm{\zeta}_1,\bm{\zeta}_2,\bm{\zeta}_3,\bm{\zeta}_4) \in\R^{8n}$ 
is the vector of Lagrange multipliers associated to the linear constraints in~\eqref{eq:modelnew} and $\rho>0$ is a penalty parameter.

Solving problem \eqref{eq:modelnew} amounts to seek the saddle point for the augmented Lagrangian, which can be done by employing the popular Alternating Direction Method of Mutipliers (ADMM)~\cite{admm}. 
The $t$-th iteration of the ADMM
reads:
\begin{align}
\label{eq:x_ADMM}
\bm{x}^{(t+1)}\;{\in}\;&\arg\min_{\bm{x}\,\in\,\R^{3n}}\mathcal{L}(\bm{x},\bm{z}^{(t)},\bm{\zeta}^{(t)};\rho)\\
\label{eq:z_ADMM}
\bm{z}^{(t+1)}\;{\in}\;&\arg\min_{\bm{z}\,\in\,\R^{8n}}\mathcal{L}(\bm{x}^{(t+1)},\bm{z},\bm{\zeta}^{(t)};\rho)\\
\label{eq:getzeta}
\bm{\zeta}^{(t+1)}\;{=}\;&\bm{\zeta}^{(t)}\;{+}\; \rho\left(\bm{\mathrm{H}}\bm{x}^{(t+1)}+\bm{\mathrm{G}}\bm{z}^{(t+1)}\right)\,.
\end{align}
Now we detail how to solve \eqref{eq:x_ADMM}-\eqref{eq:z_ADMM} and how to set the regularization parameter $\lambda$ so that the inner iterate $\bm{u}^{(t)}$ satisfies the ADP.

\subsection{The $\bm{x}$-subproblem}

Recalling that $F_1(\bm{x})=0$, we have that problem~\eqref{eq:x_ADMM} reads
\begin{equation}
\bm{x}^{(t+1)}\;{\in}\;\arg\min_{\bm{x}\,\in\,\R^{3n}}\left\{\langle\bm{\zeta}^{(t)},\bm{\mathrm{H}x}-\bm{ z}^{(t)}\rangle+\frac{\rho}{2}\|\bm{\mathrm{H}x}-\bm{ z}^{(t)}\|_2^2\right\}\,,
\end{equation}
where the constant terms can be re-adjusted so as to give
\begin{equation}\label{eq:subx}
\bm{x}^{(t+1)}\;{\in}\;\arg\min_{\bm{x}\in\R^{3n}}\left\{  \|\bm{\mathrm{H}x}\;{-}\;\bm{v}^{(t)}\|_2^2\right\}\,,\quad \mathrm{with}\quad \bm{v}^{(t)}\;{=}\;\bm{z}^{(t)}-\frac{\bm{\zeta}^{(t)}}{\rho}\,.
\end{equation}
Solving problem~\eqref{eq:subx} amounts to solving the following normal equations:
\begin{equation}\label{eq:getuw}
\bm{\mathrm{H}}^T    \bm{\mathrm{H}}    \bm{x} = \bm{\mathrm{H}}^T \bm{v}^{(t)}\,.
\end{equation}
Under the adoption of periodic boundary conditions, the matrices $\bm{\A}$, $\bm{\D_h}$ and $\bm{\D_v}$ are Block Circulant with Circulant Blocks (BCCB), and they can be thus diagonalized by means of the 2D Fast Fourier Transform (FFT). By exploiting this property, which is block-wise inherited by the matrix $\bm{\mathrm{H}}^T\bm{\mathrm{H}}$, in \cite{DLV_dtgv} the authors have shown that the update of $\bm{x}$, that is the update of $\bm{u}$ and $\bm{w}=(\bm{w_1},\bm{w}_2)$, can be obtained in a very efficient manner by three applications of the forward 2D FFT and three applications of the inverse 2D FFT. We refer the reader to \cite{DLV_dtgv} for the extended computations.

\subsection{The $\bm{z}$-subproblem}

The subproblems for the update of $\bm{z}_1$, $\bm{z}_2$, $\bm{z}_3$ and $\bm{z}_4$ can be addressed separately. Moreover, one can immediately notice that the regularization parameter $\lambda$ is only involved in the subproblem concerning $\bm{z}_1$, which, after the introduction of the auxiliary vector $\bm{q}^{(t)}=\bm{\mathrm{H}}\bm{x}^{(t+1)}+\bm{\zeta}^{(t)}/\rho$, reads:
\begin{equation}\label{eq:z1upd}
\bm{z_1}^{(t+1)}(\tau)\;{\in}\; \arg\min_{\bm{z}_1\,\in\,\R^n}\left\{\lambda\,D_{KL}(\bm{z}_1\,+\,\bm{\gamma};\bm{b})\;{+}\;\frac{1}{2}\|\bm{z}_1-\bm{q}_1^{(t)}\|_2^2\right\}\,,
\end{equation}
with $\bm{q}_1\in\R^n$ being the sub-vector of $\bm{q}$ corresponding to the first constraint in \eqref{eq:TGVKLcon}, and $\tau:=\lambda/\rho$. Notice that in \eqref{eq:z1upd} we made explicit the dependence of $\bm{z}_1$ on $\tau$, i.e. on $\lambda$. 

Problem \eqref{eq:z1upd} is separable, so that, after dropping out the constant terms, one has to tackle the solution of $n$ one-dimensional minimization problems of the form
\begin{equation}
\label{eq:zupd}
z_i^{(t+1)}(\tau)\;{=}\;\arg\min_{z_i\,\in\,\R}\bigg\{\tau[(z_i+\gamma_i)-b_i \ln(z_i+\gamma_i)  ]+\frac{1}{2}(z_i-q_i^{(t)})^2
\bigg\}.
\end{equation}
where we set $z_i:=(\bm{z}_1)_i$ and $q_i = (\bm{q}_1)_i$ to avoid heavy notations. By imposing a first order optimality condition on the convex objective function in \eqref{eq:zupd}, we get
\begin{equation}\label{eq:getz}
z_i^{(t+1)}(\tau) = \frac{1}{2}\left[-(\tau+\gamma_i-q_i^{(t)})+\sqrt{(\tau+\gamma_i-q_i^{(t)})^2+4\tau\left(\frac{q_i^{(t)}\,\gamma_i}{\tau}+b_i-\gamma_i\right)
}\,\right]\,.
\end{equation}
Hence, after setting a value for the parameter $\tau$, the variable $\bm{z}_1$ can be updated by computing \eqref{eq:getz} for all $i=1,\ldots,n$.

As previously done in \cite{BF,Teuber}, the DP can be imposed along the ADMM iterations so that the optimal $\tau$ is selected as the one solving the non-linear discrepancy equation
\begin{equation}\label{eq:getlam}
\sum_{i=1}^n F(z_i^{(t+1)}(\tau);b_i)=\frac{n}{2}\,,
\end{equation}
that is
\begin{equation}\label{eq:DPeq}
\sum_{i=1}^n \left(b_i\ln b_i-b_i\ln(z_i(\tau)+\gamma_i)+(z_i(\tau)+\gamma_i)-b_i\right)-\frac{n}{2}  =0\,.
\end{equation}
The discrepancy function in the right hand-side of \eqref{eq:DPeq} can be proven to be a convex and decreasing function of $\tau$ \cite{BF}, so that equation \eqref{eq:DPeq} admits a solution that can be sought by means of the Newton-Raphson algorithm.

The update of $\bm{z}_2$ and $\bm{z}_3$ relies on the computation of the proximal operator 
of the function $\|\cdot\|_2$ at a given vector $\bm{f}$: 
\begin{equation}
\mathrm{prox}_{c\|\cdot\|_2}(\bm{f}) = \arg\min_{\bm{g}}\left\{c\|\bm{f}\|_2\;{+}\;\frac{1}{2}\|\bm{g}-\bm{f}\|_2^2\right\}\,,
\end{equation}
where $c$ is a given constant. We have
\begin{align}\label{eq:z2upd}
\bm{z}_2^{(t+1)}\;{\in}\;&\arg\min_{\bm{z}_2\,\in\,\R^{2n}}\left\{\alpha_0\sum_{i=1}^n\|(\bm{z}_2)_i\|_2\;{+}\;\frac{\rho}{2}\|\bm{z}_2-\bm{q}_2^{(t)}\|_2^2\right\}\,,\\
\label{eq:z3upd}
\bm{z}_3^{(t+1)}\;{\in}\;&\arg\min_{\bm{z}_3\,\in\,\R^{4n}}\left\{\alpha_1\sum_{i=1}^n\|(\bm{z}_3)_i\|_2\;{+}\;\frac{\rho}{2}\|\bm{z}_3-\bm{q}_3^{(t)}\|_2^2\right\}\,,
\end{align}
with $\bm{q}_2$ and $\bm{q}_3$ being the sub-vectors of $\bm{q}$ corresponding to the second and third constraint in \eqref{eq:TGVKLcon}, respectively.

Problems \eqref{eq:z2upd} and \eqref{eq:z3upd} can be split into $n$ two-dimensional and $n$ four-dimensional subproblems, respectively, whose general form is
\begin{align}\label{eq:z2updi}
(\bm{z}_2)_i^{(t+1)}\;{\in}\;&\arg\min_{(\bm{z}_2)_i\,\in\,\R^{2}}\left\{\frac{\alpha_0}{\rho}\|(\bm{z}_2)_i\|_2\;{+}\;\frac{\rho}{2}\|(\bm{z}_2)_i-(\bm{q}_2)_i^{(t)}\|_2^2\right\}\;{=}\;\mathrm{prox}_{(\alpha_0/\rho)\|\cdot\|_2}((\bm{q}_2)_i^{(t)})\,,\\
\label{eq:z3updi}
(\bm{z}_3)_i^{(t+1)}\;{\in}\;&\arg\min_{(\bm{z}_3)_i\,\in\,\R^4}\left\{\frac{\alpha_1}{\rho}\|(\bm{z}_3)_i\|_2\;{+}\;\frac{\rho}{2}\|(\bm{z}_3)_i-(\bm{q}_3)_i^{(t)}\|_2^2\right\}\;{=}\;\mathrm{prox}_{(\alpha_1/\rho)\|\cdot\|_2}((\bm{q}_3)_i^{(t)})\,,
\end{align}
which yields (see, e.g., \cite[Chapter 6]{prox})
\begin{align}\label{eq:getz2i}
(\bm{z}_2)_i^{(t+1)}\;{=}\;&\max\left( 1\;{-}\;\frac{\alpha_0}{\rho}\frac{1}{\|(\bm{q}_2)_i^{(t)}\|_2},0 \right)(\bm{q}_2)_i^{(t)}\,,\\
\label{eq:getz3i}
(\bm{z}_3)_i^{(t+1)}\;{=}\;&\max\left( 1\;{-}\;\frac{\alpha_1}{\rho}\frac{1}{\|(\bm{q}_3)_i^{(t)}\|_2},0 \right)(\bm{q}_3)_i^{(t)}\,.
\end{align}

Finally, the $\bm{z}_4$-subproblem reads
\begin{equation}
\bm{z}_4^{(t+1)}\;{\in}\;\arg\min_{\bm{z}_4\,\in\,\Omega}\left\{\|\bm{z}_4-\bm{q}_4^{(t)}\|_2^2\right\}\,,
\end{equation}
with $\bm{q}_4$ being the subvector of $\bm{q}$ corresponding to the fourth constraint in \eqref{eq:TGVKLcon}.
The solution of the above minimization problem is given by the unique Euclidean projection of vector $\bm{q}_4^{(t)}$ on the (convex) non-negative orthant and admits the simple closed-form
expression
\begin{equation}\label{eq:getz4i}
(\bm{z}_4)_i^{(t+1)}\;{=}\;\max\left((\bm{q}_4)_i^{(t)},0\right)\,,\quad i=1,\ldots,n\,.
\end{equation}



\subsection{Updating $\alpha_0$ and $\alpha_1$}

One can easily notice that the updating step for $(\alpha_0,\alpha_1)$ can be split into the two minimization problems
\begin{align}\label{eq:upda0}
\alpha_0^{(k+1)}\;{\in}\;&\arg\min_{\alpha_0\,\in\,\R_{++}}\{f(\alpha_0;z_u(\bm{u}^{(k+1)};\bm{w}^{(k+1)}),k_0,\theta_0)\}\\
\label{eq:upda1}
\alpha_1^{(k+1)}\;{\in}\;&\arg\min_{\alpha_1\,\in\,\R_{++}}\{f(\alpha_1;z_w(\bm{w}^{(k+1)}),k_1,\theta_1)\}\,,
\end{align}
where
\begin{equation}
\label{eq:funf}
f: \alpha\in\R_{++} \to f(\alpha;\bm{a},k,\theta)\in\R\,,\quad f(\alpha;\bm{a},k,\theta) \;{=}\; \alpha \sum_{i=1}^n\|\bm{a}_i\|_2-(n+k-1)\ln\alpha+\frac{\alpha}{\theta}\,, 
\end{equation}
and $z_u$, $z_w$ are defined in \eqref{eq:zu}, \eqref{eq:zw}.
%
%
%
The function $f$ in \eqref{eq:funf} is continuous and convex, hence it admits a unique global minimizer. Therefore, the updates in \eqref{eq:upda0} and \eqref{eq:upda1} can be computed by imposing a first-order optimality condition on $f$. In formula:
\begin{equation}
f'(\alpha) \;{=}\;0\,,\quad f'(\alpha)\;{=}\;\sum_{i=1}^n\|\bm{a}_i\|_2\,-\,(n+k-1)\,\frac{1}{\alpha}\,+\,\frac{1}{\theta}=0\,,
\end{equation}
which yields
\begin{align}\label{eq:getal0}
\alpha_0^{(k+1)} \;{=}\;& (n+k_0-1) \left(\sum_{i=1}^n\|(\bm{\D u}^{(k+1)})_i-(\bm{w}^{(k+1)})_i\|_2\,+\,\frac{1}{\theta_0}\right)^{-1}\\
\label{eq:getal1}
\alpha_1^{(k+1)} \;{=}\;& (n+k_1-1) \left(\sum_{i=1}^n\|(\bm{w}^{(k+1)})_i\|_2\,+\,\frac{1}{\theta_1}\right)^{-1}\,.
\end{align}
The overall numerical procedure is summarized in Algorithm \ref{alg:2}.

\begin{algorithm}
	\caption{Alternating scheme for the TGV$^2$-KL model with automatic parameter estimation}
	\vspace{0.2cm}
	{\renewcommand{\arraystretch}{1.2}
		\renewcommand{\tabcolsep}{0.0cm}
		\vspace{-0.08cm}
		\begin{tabular}{ll}
			\textbf{inputs}:      & observed image $\,\bm{b}\in \mathbb{N}^n$, emission background $\bm{\gamma}\in\R_{++}^n$, \vspace{0.04cm}  \\
			&regularization operators $\bm{\D}\in\R^{2n \times n}$, $\bm{\mathcal{E}}\in\R^{4n \times 2n}$\vspace{0.04cm} \\
			& blur operator $\bm{\mathrm{A}}\in\R^{n\times n}$, penalty parameter $\rho\in\R_{++}$
			\vspace{0.04cm} \\
			\textbf{output}:$\;\;$     & estimated restored image $\,{\bm{u}} ^*\in \R^n$ \vspace{0.2cm} \\
		\end{tabular}
	}
	\vspace{0.1cm}
	{\renewcommand{\arraystretch}{1.2}
		\renewcommand{\tabcolsep}{0.0cm}
		\begin{tabular}{rcll}
			1. & $\quad$ & \multicolumn{2}{l}{\textbf{initialize:}$\;\;$ 
				set $\,\bm{u}^{(0)}=\bm{b}$, $\,\bm{w}^{(0)}=\bm{\D u}^{(0)}$, $\alpha_0^{(0)}$, $\alpha_1^{(0)}$  } \vspace{0.05cm}\\
			2. && \multicolumn{2}{l}{\textbf{for} $\;$ \textit{k = 0, 1, 2, $\, \ldots \,$ until convergence $\:$} \textbf{do}} \vspace{0.1cm}\\
			3. && \multicolumn{2}{l}{$\quad\;\;$ {\textbf{for} $\;$ \textit{t = 0, 1, 2, $\, \ldots \,$ until convergence $\:$}\textbf{do}} } \vspace{0.05cm} \\
			4. && $\quad\qquad\;\;\bf{\cdot}$ {compute $\:\bm{u}^{(t+1)},\bm{w}^{(t+1)}$}   & 
			by solving the linear system \eqref{eq:getuw} in $\bm{x}$\vspace{0.05cm} \\
			5. && $\quad\qquad\;\;\bf{\cdot}$ {update $\:\,\;\;\lambda$ }   & 
			by solving the non-linear equation \eqref{eq:getlam} \vspace{0.05cm} \\
			6. && $\quad\qquad\;\;\bf{\cdot}$ {compute $\:\,\bm{z}_1^{(t+1)}$ }   & 
			by \eqref{eq:getz} \vspace{0.05cm} \\
			7. && $\quad\qquad\;\;\bf{\cdot}$ {compute $\:\,\bm{z}_2^{(t+1)}, \bm{z}_3^{(t+1)}, \bm{z}_4^{(t+1)}$ }   & 
			by \eqref{eq:getz2i}, \eqref{eq:getz3i}, \eqref{eq:getz4i}, respectively \vspace{0.05cm} \\
			8. && $\quad\qquad\;\;\bf{\cdot}$ {update $\:\,\bm{\zeta}^{(t+1)}$ }   & by \eqref{eq:getzeta}\\
			9. && \multicolumn{2}{l}{$\quad\;\;$ {\textbf{end for} } } \vspace{0.05cm} \\
			10. && \multicolumn{2}{l}{$\quad \;\;$ ${\bm{u}}^{(k+1)} = \bm{u}^{(t+1)},\, {\bm{w}}^{(k+1)} = \bm{w}^{(t+1)}$}
			\\
			11. && $\quad\;\;\bf{\cdot}$ {compute $\:\alpha_0^{(k+1)}$}   & 
			by \eqref{eq:getal0}\vspace{0.05cm} \\
			12. && $\quad\;\;\bf{\cdot}$ {compute $\:\alpha_1^{(k+1)}$}   & 
			by \eqref{eq:getal1}\vspace{0.05cm} \\
			13. && \multicolumn{2}{l}{\textbf{end for} $\;$  } \vspace{0.1cm}\\
			14. && \multicolumn{2}{l}{ $\; {\bm{u}}^*=\bm{u}^{(k+1)}$  } \vspace{0.1cm}
		\end{tabular}
	}
	\label{alg:2}
\end{algorithm}

\section{Computational results\label{sec:ex}}

In this section we evaluate the performance of the proposed strategy for the automatic estimation of the parameters $\alpha_0$ and $\alpha_1$. 
As our proposal makes use of the ADP as criterion for the estimation of the regularization parameter $\lambda$, we first evaluate the robustness of the employed principle when solving the TGV$^2$-KL model with fixed $\alpha_0$ and $\alpha_1$.

The quality of the restoration $\bm{u}^*$ with respect to the original image $\bm{u}$ will be measured in terms of the Improved Signal-to-Noise Ratio (ISNR),
\begin{equation}
\mathrm{ISNR}(\bm{b},\bm{u},\bm{u}^*) \;{:=}\; 10\,\log_{10}\frac{\|\bm{b}-\bm{u}\|_2^2}{\|\bm{u}^*-\bm{u}\|_2^2}\,, 
\end{equation}
and of the Structural Similarity Index (SSIM) \cite{ssim}. The output $\bm{u}^*$ of the procedure outlined in Algorithm~\ref{alg:2} is compared with the target image $\bar{\bm{u}}$, which throughout this section denotes the solution of the TGV$^2$-KL model obtained by manually fixing the parameters so as to achieve the highest ISNR value\footnote{The \emph{highest} ISNR has to be intended as the highest value achieved for the parameters varying on selected grids with accuracy $10^{-3}$.}. We remark that when solving the TGV$^2$-KL model 
with fixed parameters $\alpha_0,\alpha_1$, the regularization parameter $\lambda$ is somehow superfluous as $\alpha_0,\alpha_1$ can be tuned so as to get an `optimal' balance between the KL fidelity term and the TGV$^2$ regularizer. Thus, $\bar{\bm{u}}$ has been computed by keeping $\lambda=1$. 




The ADMM with the adaptive adjustment of the regularization parameter $\lambda$ has been implemented starting from the implementation of the method used in \cite{DLV_dtgv,DLV_tgv}, available as a MATLAB code at \texttt{https://github.com/diserafi/respond}.
The outer iterations of the alternating scheme as well as the inner ADMM iterations are stopped as soon as
\begin{equation}
\delta_{\bm{u}}^{(k)} \;{:=}\; \frac{\|\bm{u}^{(k)}-\bm{u}^{(k-1)}\|_2}{\|\bm{u}^{(k-1)}\|_2}\;{<}\;10^{-5}\,.
\end{equation}
The ADMM penalty parameter $\rho$ is set so as to fasten the convergence. In most of the examples, it is $\rho=0.1$.

The experiments have been performed using MATLAB R2019b under Windows 10 on an ASUS PC with a 2.50 GHz Intel Core i7 6500U processor and 8 GB of RAM.

\subsection{(I)ADP for the TGV$^2$-KL image restoration model}

We assess the robustness of the ADP employed for the automatic estimation of $\lambda$ when the parameters $\alpha_0$ and $\alpha_1$ are fixed. More specifically, we are going to show that:
\begin{itemize}
	\item[-] in the \emph{a-posteriori} scenario, the discrepancy equation in \eqref{eq:DP2} admits a unique solution;
	\item[-] when applying the ADP along the ADMM iterations, the output solution is very close to (if not practically the same as) the one obtained when applying the criterion a posteriori.
\end{itemize}
In the following, we refer to the iterated version of the ADP as IADP.

We consider the restoration of three different test images, namely \texttt{phantom} in Figure \ref{fig:ph_or}, which is a predominantly smooth image with edges, \texttt{penguin} in Figure \ref{fig:peng_or}, which is a natural image with smooth and piece-wise constant features, and the medical test image \texttt{brain} in Figure \ref{fig:br_or}, presenting fine smooth details.

\begin{figure}
	\centering
	\begin{subfigure}[t]{0.32\textwidth}
		\centering
		\includegraphics[height=4.2cm]{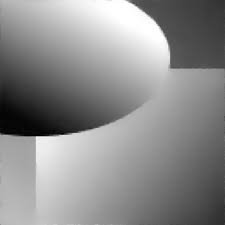}\caption{}
		\label{fig:ph_or}
	\end{subfigure}
	\begin{subfigure}[t]{0.32\textwidth}
		\centering
		\includegraphics[height=4.2cm]{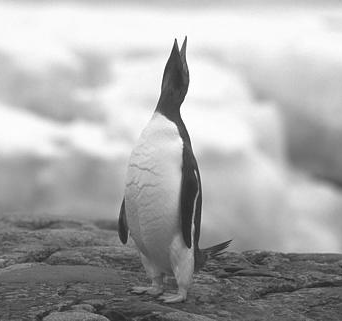}
		\caption{}
		\label{fig:peng_or}
	\end{subfigure}
	\begin{subfigure}[t]{0.32\textwidth}
		\centering
		\includegraphics[height=4.2cm]{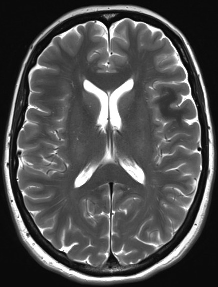}\caption{}
		\label{fig:br_or}
	\end{subfigure}
	\caption{Gray scale test images considered in the experiments: \texttt{phantom} ($225\times 225$) (a), \texttt{penguin} ($321\times 342$) (b), and \texttt{brain} ($287\times 218$) (c).}
	\label{fig:or}
\end{figure}

The original images, with pixel values between 0 and 1, were scaled by a factor $\kappa=50$ indicating the maximum number of photons allowed to hit the image domain. Then, a Gaussian blur, generated by the MATLAB routine \texttt{fspecial} with parameters \texttt{band} = 5 and \texttt{sigma} = 1, was applied to the scaled images. The \texttt{band} parameter represents the side length (in pixels) of the square support of the kernel, whereas \texttt{sigma} is the standard deviation of the circular, zero-mean, bivariate Gaussian pdf representing the kernel in the continuous setting. A constant background $\gamma=2\times 10^{-3}$ has been added to the blurred images, which have been further degraded with a Poisson noise.

For each test image we solved the TGV$^2$-KL model for different values of $\lambda$ and fixed parameters $\alpha_0,\alpha_1$. 
For each restored solution $\bm{u}^*(\lambda)$, we computed the corresponding value of the discrepancy function in \eqref{eq:def_D} and checked whether it intersected the horizontal line corresponding to $\Delta = n/2$ - see \eqref{eq:DP2}. The behavior of the discrepancy function over the selected grids of $\lambda$-values is reported on the left column of Figure \ref{fig:discr} for the three test images. One can notice that the experimental discrepancy function is decreasing and admits a unique intersection with the horizontal dashed blue line, representing the discrepancy value $\Delta = n/2$. The value of $\lambda$ satisfying the ADP is indicated by the vertical solid magenta line. In the discrepancy plots, we also report with a vertical dashed green line the value of $\lambda$ achieved by the IADP along the ADMM iterations. One can observe that the ADP and the IADP detect very close values of $\lambda$.

As a further analysis of the robustness of the IADP, we also report 
on the right column of Figure~\ref{fig:discr} the behavior of the parameter $\lambda$ along the ADMM iterations, whence one can observe that $\lambda$ starts stabilizing after 100 iterations.

\begin{figure}
	\centering
	\begin{tabular}{cc}
		\includegraphics[height=4.2cm]{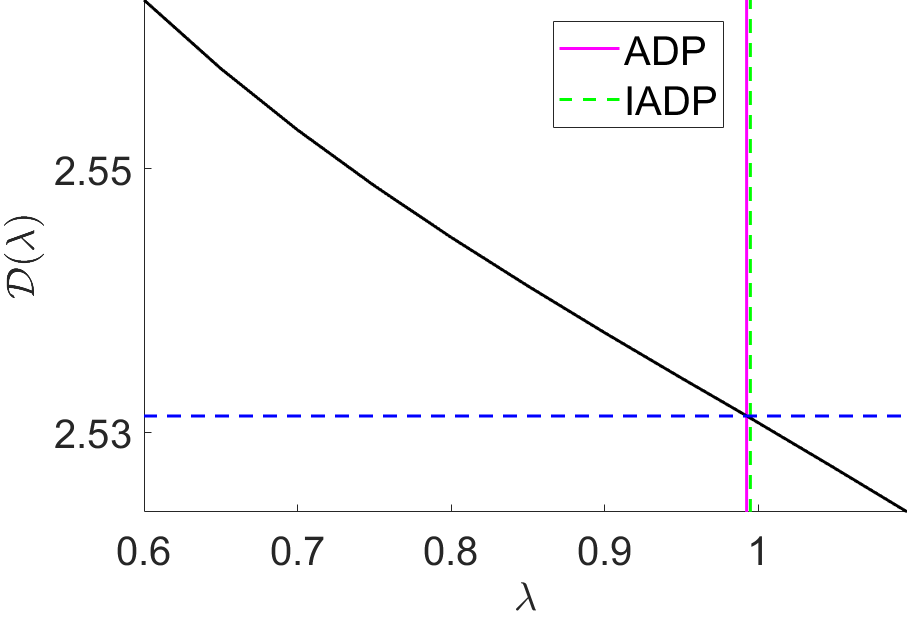}
		& \includegraphics[height=4.2cm]{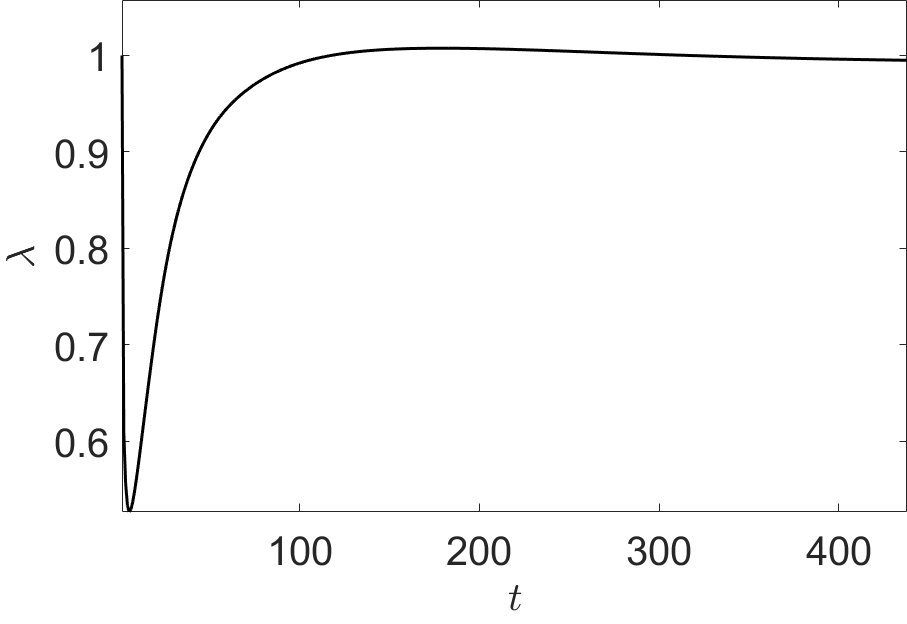}
		\\
		\multicolumn{2}{c}{(a) \texttt{phantom}}\\[12pt]
		\includegraphics[height=4.2cm]{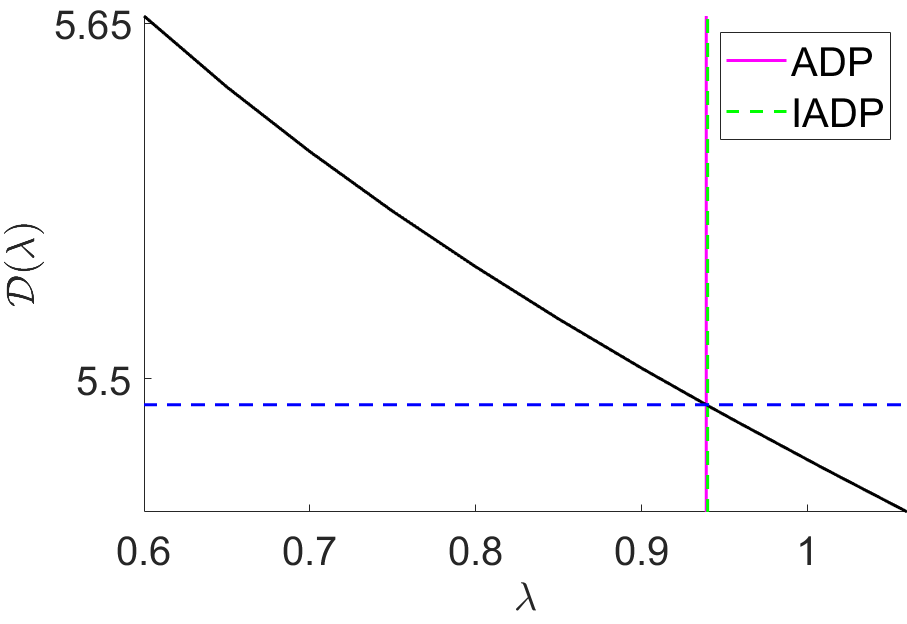}
		& \includegraphics[height=4.2cm]{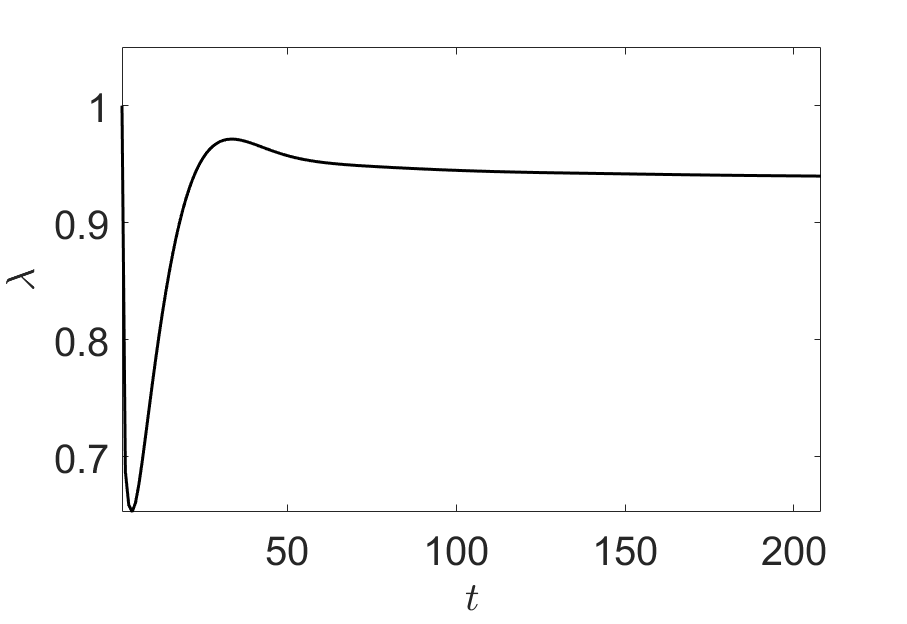}
		\\
		\multicolumn{2}{c}{(b) \texttt{penguin}}\\[4pt]
		\\
		\includegraphics[height=4.2cm]{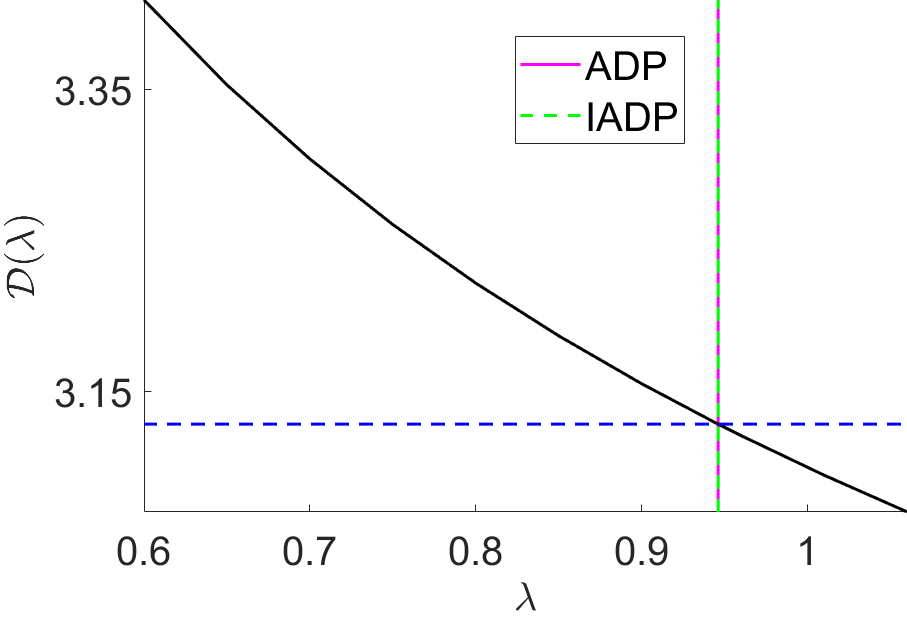}
		& \includegraphics[height=4.2cm]{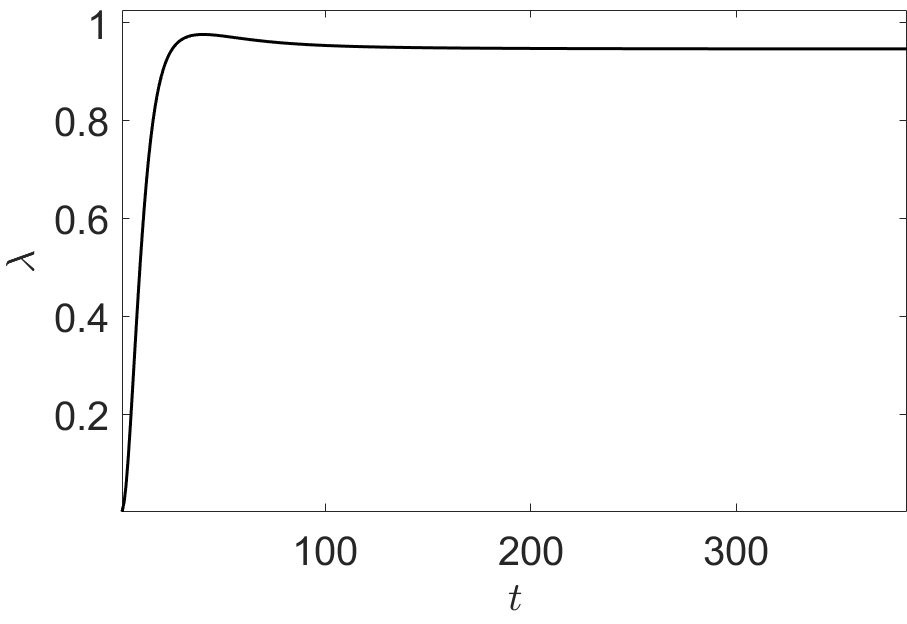}
		\\
		\multicolumn{2}{c}{(c) \texttt{brain}}
	\end{tabular}
	\caption{Discrepancy curve (left) and behavior of $\lambda$ along the ADMM iterations when employing the IADP (right) for the three test images. Note that in the left pictures the values on the $y$ axis are divided by $10^{4}$. 
		\label{fig:discr}}
\end{figure}

Finally, we highlight that the value of $\lambda$ selected by the ADP and the IADP for the test images is approximately 1. This is not surprising as the parameters $\alpha_0,\alpha_1$ have been set as the ones returning the highest ISNR when solving the TGV$^2$-KL model with $\lambda=1$. In other words, in the considered counting regime, the (I)ADP automatically selects $\lambda$ so as to maximize the ISNR.

\subsection{Results of joint estimation of the images and the parameters}

Relying on the results shown in the previous section, which show the robustness of the IADP for the TGV$^2$-KL variational model, we employ the procedure in Algorithm~\ref{alg:2} for the joint estimation of the unknown images and parameters. More specifically, we address the restoration of the three test images
corrupted by Gaussian blur with $\texttt{band}=5$ and $\texttt{sigma}=1$, with the addition of a constant background $\gamma=2\times 10^{-3}$, and further degradated by Poisson noise with different mid and high counting factors, namely $\kappa=30,50,100,500$.
For each test, the initial guess $\bm{u}^{(0)}$ and the parameters $k_0,\theta_0,k_1,\theta_1$ were computed as detailed in Section~\ref{sec:init}.

In Table \ref{tab:vals 1}, for each test image \texttt{phantom} and counting factor $\kappa$, we report the ISNR and SSIM values of the initial TV-regularized solution $\bm{u}^{(0)}$, of the image $\bm{u}^*$ restored by means of Algorithm~\ref{alg:2}, and of the target image $\bar{\bm{u}}$. One can notice that the quality measures corresponding to $\bm{u}^*$ are generally close to those related to the target image $\bar{\bm{u}}$. This behavior reflects the accuracy of the estimates $\alpha_0^*$ and $\alpha_1^*$, which can be compared with the optimal values $\bar{\alpha}_0,\bar{\alpha}_1$ corresponding to $\bar{\bm{u}}$. However, we highlight that $(\alpha_0^{*},\alpha_1^{*})$ and $(\bar{\alpha}_0,\bar{\alpha}_1)$ can only be compared after dividing the former by $\lambda^*$, i.e. the value of $\lambda$ obtained at the end of the alternating scheme. 

{
	\setlength{\tabcolsep}{6pt}
	\renewcommand{\arraystretch}{1.2}	
	\begin{table}
		\small
		\centering
		\begin{tabular}{c|c|c|c|c|c|c|c}
			\hline
			&&\multicolumn{2}{c|}{$\bm{u}_0$}&\multicolumn{2}{c|}{$\bm{u}^*$}&\multicolumn{2}{c}{$\bar{\bm{u}}$}\\
			\hline
			&$\kappa$&ISNR&SSIM&ISNR&SSIM&ISNR&SSIM\\
			\hline
			\multirow{4}{*}{{\rotatebox[origin=c]{90}{\texttt{phantom}}}}&{30}&15.7627&0.5843&17.6337&0.7713&17.8140&0.7937\\
			&{50}&14.2667&0.6183&16.7371&0.8157&16.9036&0.8227\\
			&{100}&12.0545&0.6704&15.4684&0.8394&15.9207&0.8419
			\\
			&{500}&9.8265&0.7660&12.9189&0.8713&  13.1144&0.8720\\
			\hline
			\multirow{4}{*}{{\rotatebox[origin=c]{90}{\texttt{penguin}}}}&{30}&12.3237&0.4241&12.7870&0.5554&12.8399&0.5529\\
			&{50}&11.0319&0.4482&11.4996&0.5713&11.5429&0.5823\\
			&{100}&9.1828&0.4954&9.6342&0.6096&9.8388&0.6100\\
			&{500}& 5.2980&0.6042&5.9338&0.6993&6.0306&0.7018 \\
			\hline
			\multirow{4}{*}{{\rotatebox[origin=c]{90}{\texttt{brain}}}}&30&4.0738&0.5282&4.5550&0.5562&4.9683& 0.5574\\
			&{50}&3.1300&0.5587&3.6589&0.5909&4.0398&0.5907\\
			&100&1.9255&0.5946&2.4887&0.6215&3.3358&0.6288\\
			&500&1.6246&0.6825&2.4086&0.7067&2.8902&0.7088\\
			\hline
		\end{tabular}
		\caption{Quality measures, for the three test images and different photon counting regimes, of the TV-regularized initialization $\bm{u}^{(0)}$ (left), the output $\bm{u}^*$ of Algorithm~\ref{alg:2} (middle), and the restored target image $\bar{\bm{u}}$ obtained by manually tuning $\alpha_0,\alpha_1$ in the TGV$^2$-KL model so as to maximize the ISNR.\label{tab:vals 1}}
	\end{table}
}

In Table \ref{tab:vals 2}, we report the initial estimates $\alpha_0^{(0)}$, $\alpha_1^{(0)}$ computed as detailed in Section~\ref{sec:init}, the output estimates $\alpha_0^*$, $\alpha_1^*$ divided by $\lambda^*$, and the target $\bar{\alpha}_0$, $\bar{\alpha}_1$ values. The estimates obtained by our alternating scheme are mostly close to the target parameters, with the more significant difference arising in the lowest counting regime,
which corresponds to the scenario in which the IADP presents a weaker behavior.

{
	\setlength{\tabcolsep}{6pt}
	\renewcommand{\arraystretch}{1.2}	
	\begin{table}
		\small
		\centering
		\begin{tabular}{c|c|c|c|c|c|c|c}
			\hline
			&$\kappa$&$\alpha_0^{(0)}$&$\alpha_1^{(0)}$&$\alpha_0^*/\lambda^*$&$\alpha_1^*/\lambda^*$&$\bar{\alpha}_0$&$\bar{\alpha}_1$\\
			\hline
			\multirow{4}{*}{{\rotatebox[origin=c]{90}{\texttt{phantom}}}}&{30}&0.1453& 2.8596&0.1725 &2.1626&0.2056 &1.6667\\
			&{50}&0.1124&1.7896&0.1615 &0.5857&0.1620& 0.7790\\
			&100&0.1007&0.3391&0.1570 &0.4958&0.1301&0.4767\\
			&{500}&0.0347&0.1870&0.0763 &0.3771&0.0550&0.2786\\
			\hline
			\multirow{4}{*}{{\rotatebox[origin=c]{90}{\texttt{penguin}}}}&{30}&0.1305&2.1357&0.1609&0.4261 &0.1571&0.4714 \\
			&{50}&0.1009&1.2118&0.1201 & 0.4449&0.1157 &0.3102\\
			&{100}&0.0711&0.5775&0.0752&0.1313&0.0702&0.1921\\
			&{500}&0.0314&0.1018&0.0289 &0.0570&0.0251 &0.0413\\
			\hline
			\multirow{4}{*}{{\rotatebox[origin=c]{90}{\texttt{brain}}}}&30&0.2030&0.5303&0.1364 &0.2206&0.0833 &0.1021\\
			&{50}&0.1532&0.3163&0.1059 &0.1387&0.0643 &0.0786\\
			&100&0.1042&0.1593&0.0855 &0.1287&0.0412 &0.0532\\
			&500&0.0288&0.0370&0.0283 &0.0488&0.0176 &0.0198\\
			\hline
		\end{tabular}
		\caption{Parameters $\alpha_0,\alpha_1$ estimated starting from the TV-regularized initialization (left), from the output restoration of the proposed Algorithm \ref{alg:2} (middle), and manually selected so as to maximize the ISNR of the output restoration of the TGV$^2$-KL model.\label{tab:vals 2}}
	\end{table}
}

\begin{figure}
	\setlength{\tabcolsep}{1.8pt}
	\begin{tabular}{ccccc}
		&    $\bm{b}$&$\bm{u}^{(0)}$&$\bm{u}^*$&$\bar{\bm{u}}$\\
		\raisebox{1.7cm}{\rotatebox[origin=c]{90}{$\kappa=30$}}&\includegraphics[width=3.7cm]{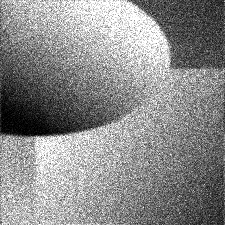}&\includegraphics[width=3.7cm]{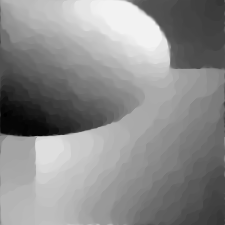}&\includegraphics[width=3.7cm]{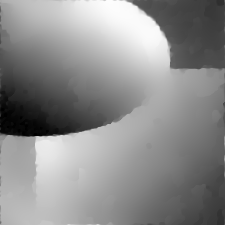}&\includegraphics[width=3.7cm]{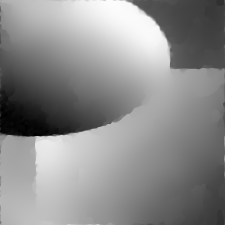}\\
		\raisebox{1.7cm}{\rotatebox[origin=c]{90}{$\kappa=50$}}&\includegraphics[width=3.7cm]{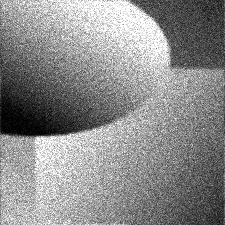}&\includegraphics[width=3.7cm]{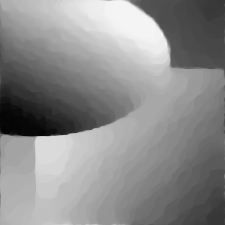}&\includegraphics[width=3.7cm]{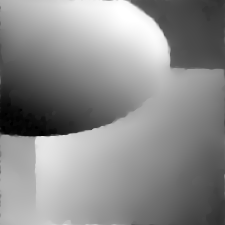}&\includegraphics[width=3.7cm]{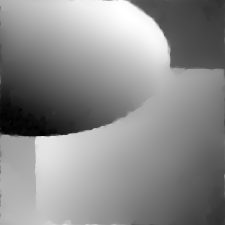}\\
		\raisebox{1.8cm}{\rotatebox[origin=c]{90}{$\kappa=100$}}&\includegraphics[width=3.7cm]{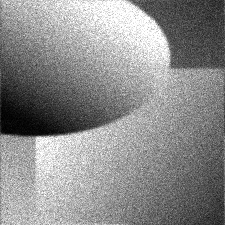}&\includegraphics[width=3.7cm]{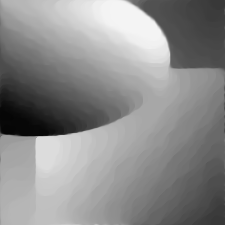}&\includegraphics[width=3.7cm]{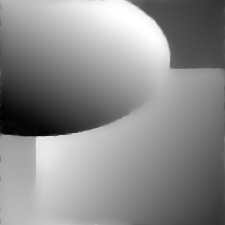}&\includegraphics[width=3.7cm]{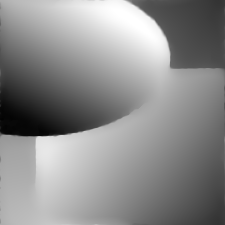}\\
		\raisebox{1.8cm}{\rotatebox[origin=c]{90}{$\kappa=500$}}&\includegraphics[width=3.7cm]{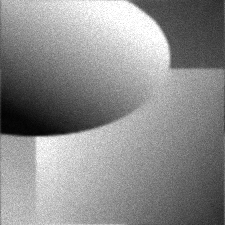}&\includegraphics[width=3.7cm]{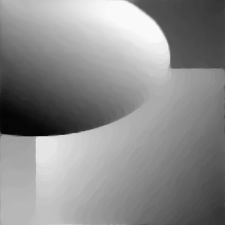}&\includegraphics[width=3.7cm]{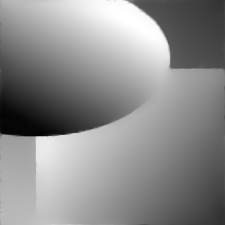}&\includegraphics[width=3.7cm]{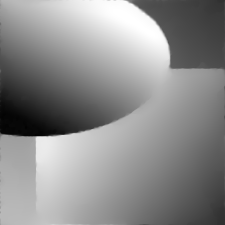}
	\end{tabular}
	\caption{From left to right: observed data $\bm{b}$, initial guess $\bm{u}^{(0)}$, output $\bm{u}^*$ of the proposed procedure, and target restoration $\bar{\bm{u}}$ for the test image \texttt{phantom} in different counting regimes.\label{fig:ph_rec}}
\end{figure}


Some specific comments on specific test images are in order. The good performance of the alternating scheme applied to \texttt{penguin} also in the lowest counting regime can be ascribed to the limited number of zeros arising in the acquisition of this image, which is significantly smaller than the number of zeros in the test image \texttt{phantom} with the same counting factor. Thus, for \texttt{penguin} the IADP is capable of selecting the regularization parameter in a better way. This behavior seems to be confirmed by the results on the image \texttt{brain}, which has a large number of zero entries. In this case, the ISNR and SSIM values in Table~\ref{tab:vals 1} suggest that the restored image $\bm{u}^*$ is not as close as to the target image $\bar{\bm{u}}$. Accordingly, the estimated parameters reported in Table~\ref{tab:vals 2} are generally larger than the target ones. However, from Figure~\ref{fig:brain_rec} one can observe that this does not yield TV-type restorations, but the output $\bm{u}^*$ is smoother than the target $\bar{\bm{u}}$. This suggests that the proposed strategy suffers from the typical oversmoothing caused by a non-optimal selection of $\lambda$.

\begin{figure}
	\setlength{\tabcolsep}{1.8pt}
	\begin{tabular}{ccccc}
		&    $\bm{b}$&$\bm{u}^{(0)}$&$\bm{u}^*$&$\bar{\bm{u}}$\\
		\raisebox{1.6cm}{\rotatebox[origin=c]{90}{$\kappa=30$}}&\includegraphics[width=3.7cm]{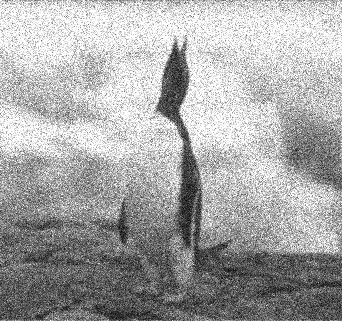}&\includegraphics[width=3.7cm]{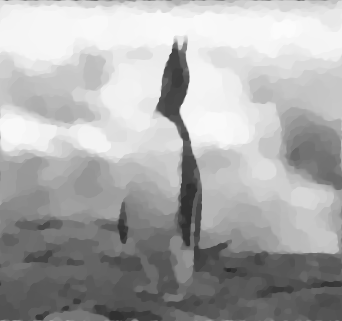}&\includegraphics[width=3.7cm]{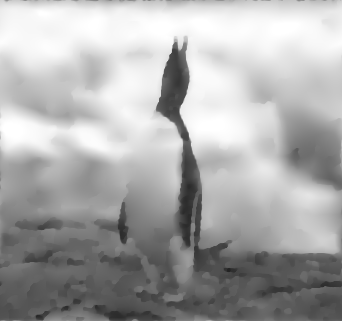}&\includegraphics[width=3.7cm]{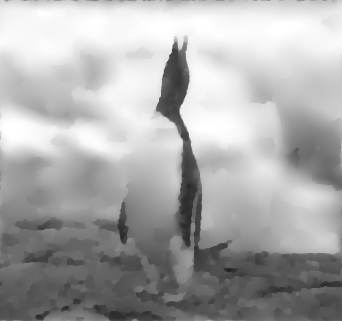}\\
		\raisebox{1.6cm}{\rotatebox[origin=c]{90}{$\kappa=50$}}&\includegraphics[width=3.7cm]{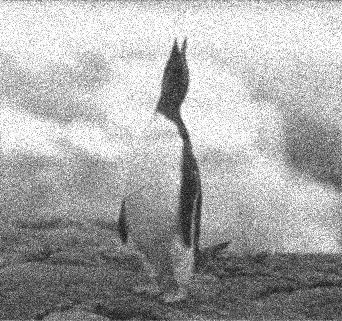}&\includegraphics[width=3.7cm]{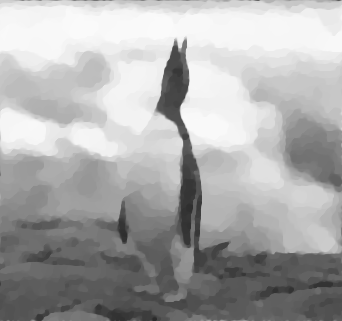}&\includegraphics[width=3.7cm]{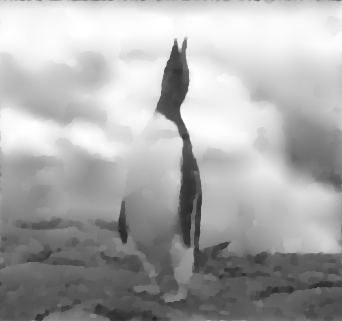}&\includegraphics[width=3.7cm]{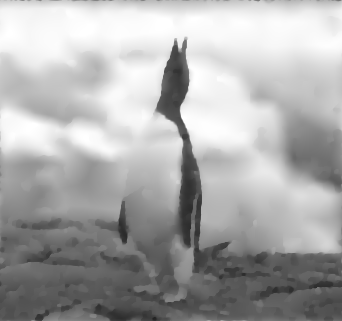}\\
		\raisebox{1.6cm}{\rotatebox[origin=c]{90}{$\kappa=100$}}&\includegraphics[width=3.7cm]{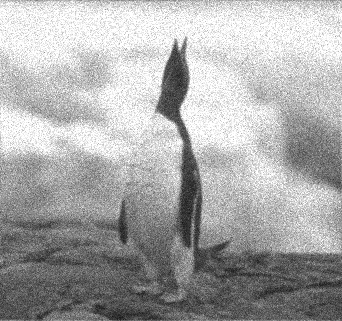}&\includegraphics[width=3.7cm]{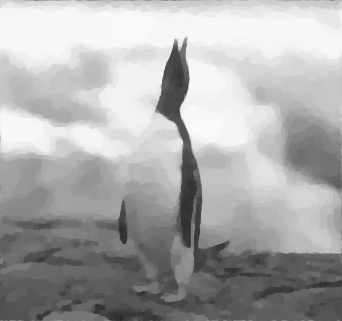}&\includegraphics[width=3.7cm]{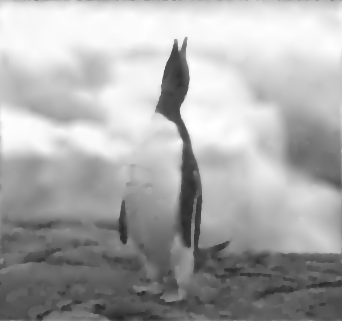}&\includegraphics[width=3.7cm]{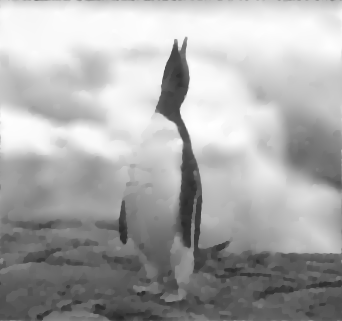}\\
		\raisebox{1.6cm}{\rotatebox[origin=c]{90}{$\kappa=500$}}&\includegraphics[width=3.7cm]{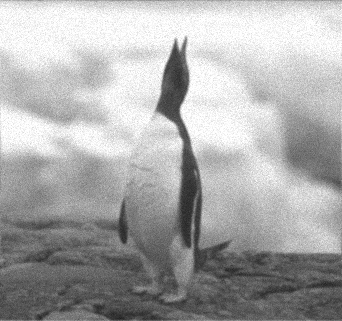}&\includegraphics[width=3.7cm]{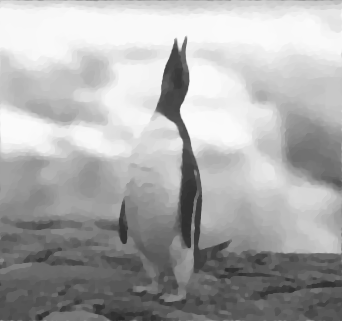}&\includegraphics[width=3.7cm]{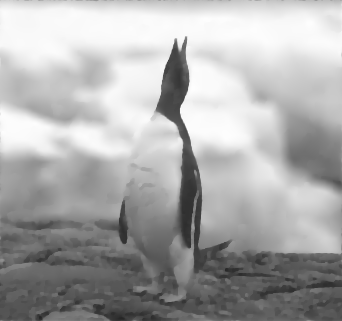}&\includegraphics[width=3.7cm]{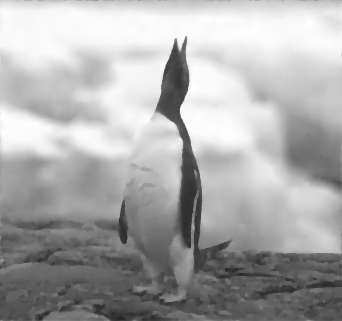}
	\end{tabular}
	\caption{From left to right: observed data $\bm{b}$, initial guess $\bm{u}^{(0)}$, output $\bm{u}^*$ of the proposed procedure, and target restoration $\bar{\bm{u}}$ for the test image \texttt{penguin} different counting regimes.\label{fig:peng_rec}}
\end{figure}


\begin{figure}
	\setlength{\tabcolsep}{1.8pt}
	\begin{tabular}{ccccc}
		&    $\bm{b}$&$\bm{u}^{(0)}$&$\bm{u}^*$&$\bar{\bm{u}}$\\
		\raisebox{2.3cm}{\rotatebox[origin=c]{90}{$\kappa=30$}}&\includegraphics[width=3.7cm]{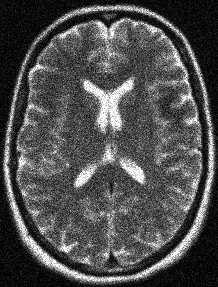}&\includegraphics[width=3.7cm]{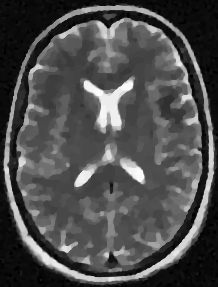}&\includegraphics[width=3.7cm]{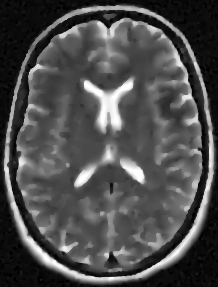}&\includegraphics[width=3.7cm]{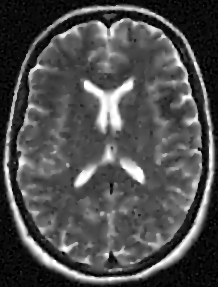}\\
		\raisebox{2.3cm}{\rotatebox[origin=c]{90}{$\kappa=50$}}&\includegraphics[width=3.7cm]{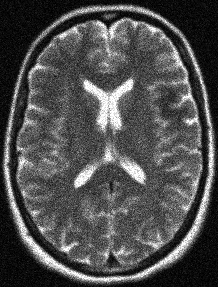}&\includegraphics[width=3.7cm]{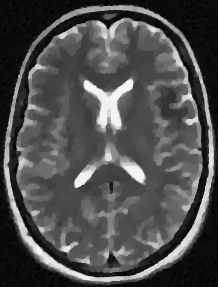}&\includegraphics[width=3.7cm]{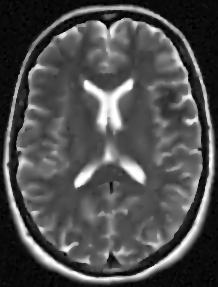}&\includegraphics[width=3.7cm]{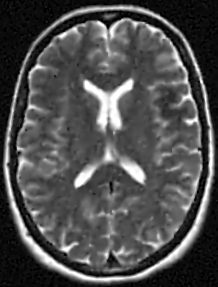}\\
		\raisebox{2.3cm}{\rotatebox[origin=c]{90}{$\kappa=100$}}&\includegraphics[width=3.7cm]{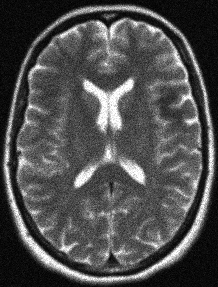}&\includegraphics[width=3.7cm]{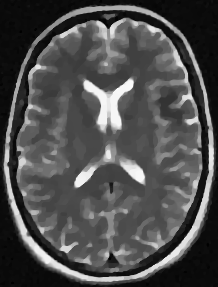}&\includegraphics[width=3.7cm]{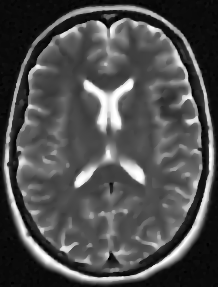}&\includegraphics[width=3.7cm]{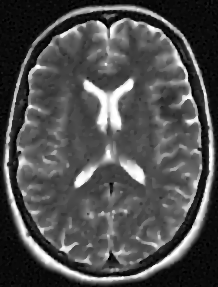}\\
		\raisebox{2.3cm}{\rotatebox[origin=c]{90}{$\kappa=500$}}&\includegraphics[width=3.7cm]{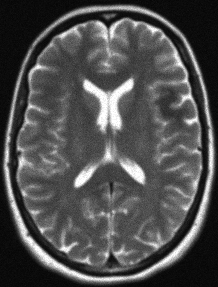}&\includegraphics[width=3.7cm]{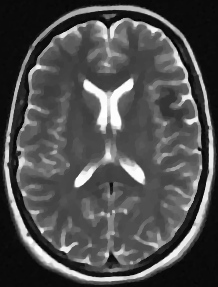}&\includegraphics[width=3.7cm]{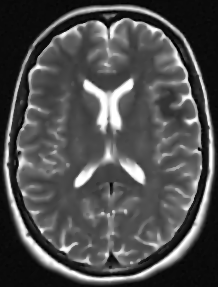}&\includegraphics[width=3.7cm]{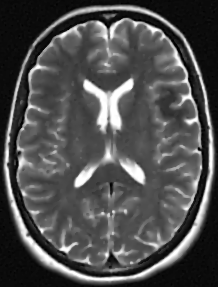}
	\end{tabular}
	\caption{From left to right: observed data $\bm{b}$, initial guess $\bm{u}^{(0)}$, output $\bm{u}^*$ of the proposed procedure, and target restoration $\bar{\bm{u}}$ for the test image \texttt{brain} in different counting regimes.\label{fig:brain_rec}}
\end{figure}


Finally, in Figure~\ref{fig:conv} we provide numerical evidence of the convergence of the overall procedure. For the image \texttt{penguin} with counting factor $\kappa=50$, we monitor the behavior of the ISNR and SSIM values, of the parameters $\alpha_0,\alpha_1$ (divided by the value of $\lambda$ achieved at the $k$-th iteration of Algorithm~\ref{alg:2}), and of the relative change along the outer iterations. We note that all these quantities stabilize after few iterations.

\begin{figure}
	\centering
	\begin{subfigure}[t]{0.4\textwidth}
		\centering
		\includegraphics[width=5.2cm]{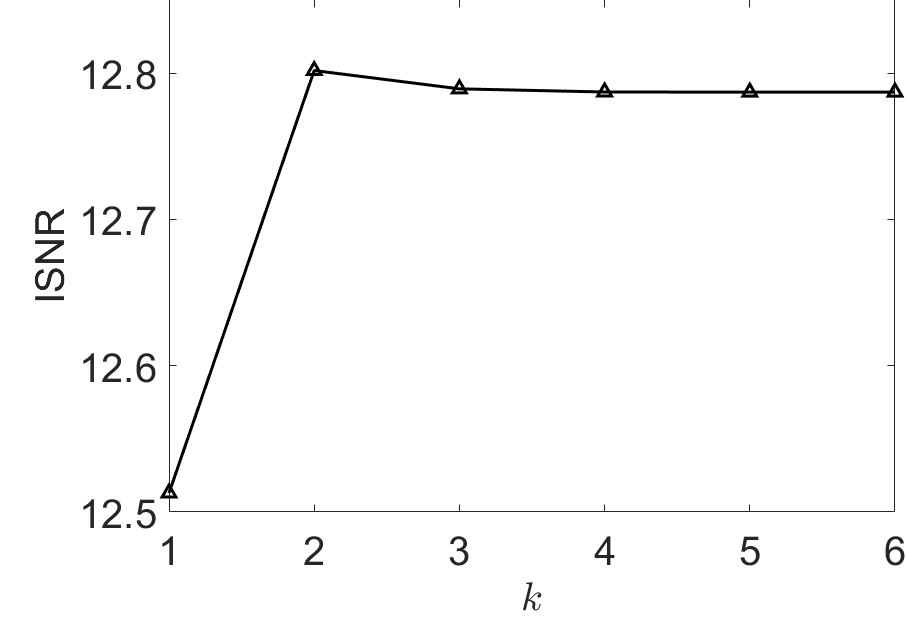}
		\caption{}
		\label{fig:peng_isnr}
	\end{subfigure}
	\begin{subfigure}[t]{0.4\textwidth}
		\centering
		\includegraphics[width=5.2cm]{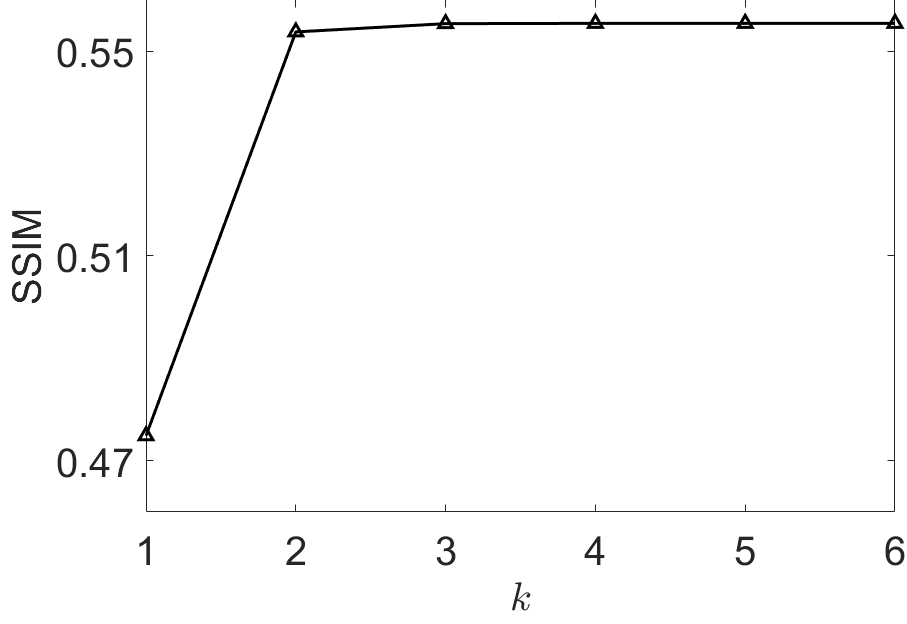}
		\caption{}
		\label{fig:peng_ssim}
	\end{subfigure}\\
	\begin{subfigure}[t]{0.32\textwidth}
		\centering
		\includegraphics[width=5.2cm]{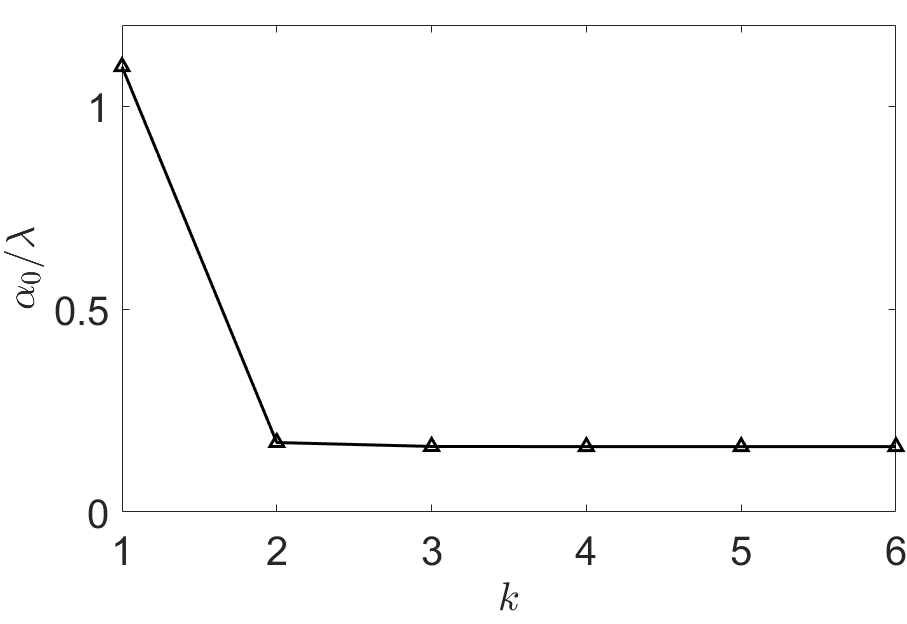}
		\caption{}
		\label{fig:peng_a0}
	\end{subfigure}
	\begin{subfigure}[t]{0.32\textwidth}
		\centering
		\includegraphics[width=5.2cm]{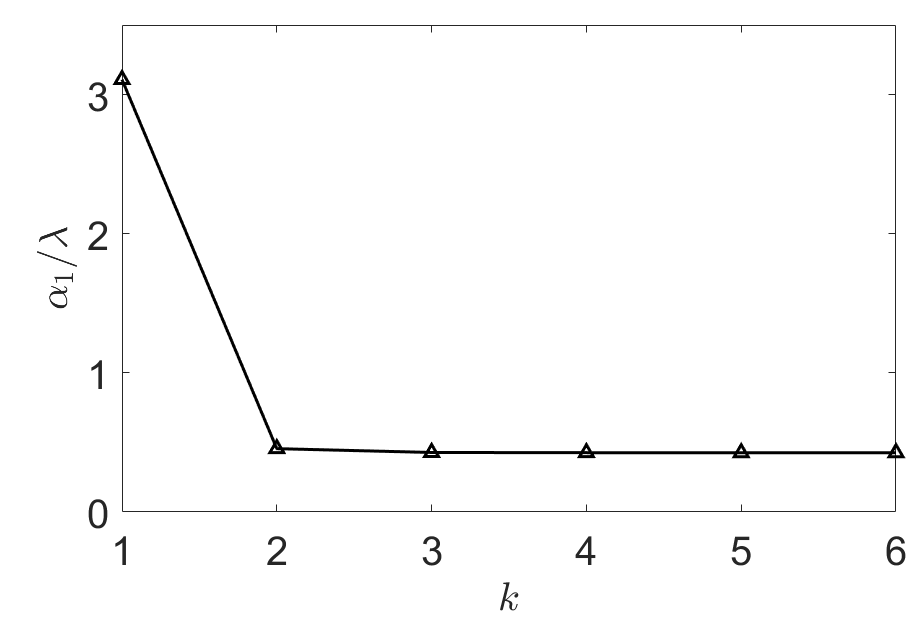}
		\caption{}
		\label{fig:peng_a1}
	\end{subfigure}
	\begin{subfigure}[t]{0.32\textwidth}
		\centering
		\includegraphics[width=5.2cm]{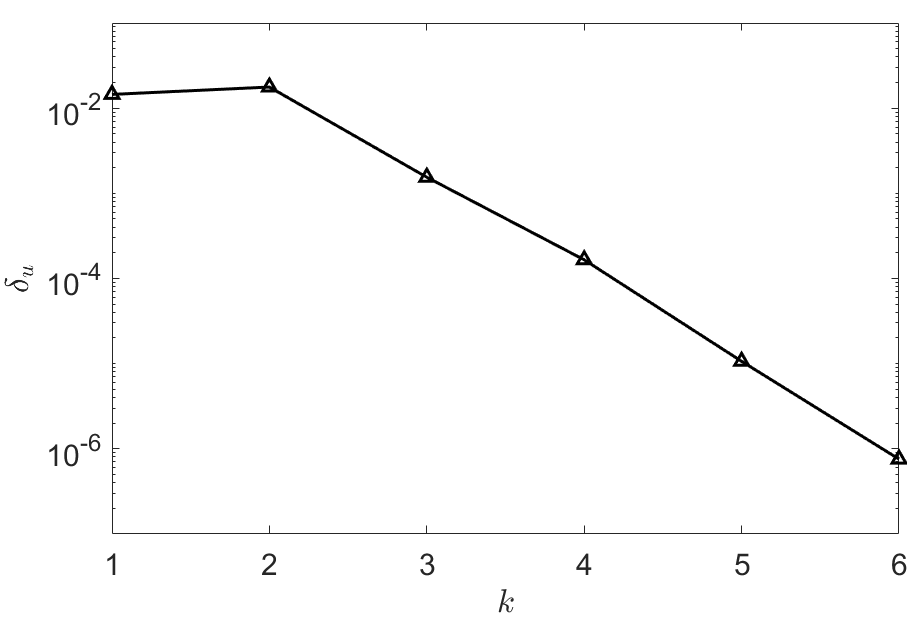}
		\caption{}
		\label{fig:peng_itr}
	\end{subfigure}
	\caption{Monitored quantities along the iterations of the alternating scheme for the restoration of \texttt{penguin} with $\kappa=50$.}
	\label{fig:conv}
\end{figure}

\section{Conclusion\label{sec:concl}}

We presented an automatic procedure for the selection of the parameters in the TGV$^2$-KL variational model for the restoration of images under Poisson noise corruption. The proposed strategy relies on a hierarchical Bayesian formulation which naturally allows us to design a probabilistic grounded alternating scheme for the joint estimation of the original uncorrupted image and the unknown parameters. The numerical procedure takes advantage of a classical discrepancy principle for the estimation of the regularization parameter $\lambda$, which is artificially incorporated in the final model so as to increase the robustness of the overall method. Numerical experiments show that the introduced machinery is capable of returning high-quality restorations that are very close to the ones achievable by manually tuning the unknown parameters in the TGV$^2$ regularization term. However, we observed that in presence of low-counting regimes or when the number of zeros in the acquisition is particularly large, the alternating method can suffer from the typical weakness of the ADP.

As a future work, we aim to investigate the interplays of the proposed procedure with different instances of the general DP, by giving a closer look at theoretical guarantees of existence and uniqueness for the solution of the discrepancy equation.

\bigskip

\noindent \textbf{Acknowledgments.} This research was partially supported by the Gruppo Nazionale per il Calcolo Scientifico of the Istituto Nazionale di Alta Matematica (GNCS-INdAM) and by the VALERE Program of the University of Campania `Luigi Vanvitelli'. 

\bigskip
\bigskip

\bibliographystyle{plain}
\bibliography{references}

\end{document}